\documentclass{amsart}

\usepackage{graphicx}
\usepackage{tikz}
\usetikzlibrary{matrix,arrows,decorations.pathmorphing}
\usepackage{tikz-cd}
\usepackage{hyperref}
\usepackage{amsthm}
\usepackage[mathscr]{euscript}
\usepackage{amssymb}
\usepackage{mathrsfs}
\usepackage{mathtools}
\usepackage{comment}
\usepackage{enumitem}

\numberwithin{equation}{section}

\newtheorem{theorem}{Theorem}[section]
\newtheorem{lemma}[theorem]{Lemma}
\newtheorem{corollary}[theorem]{Corollary}
\newtheorem{question}[theorem]{Question}
\newtheorem*{conjecture*}{Conjecture}
\newtheorem*{convention*}{Convention}
\newtheorem*{question*}{Question}

\theoremstyle{definition}
\newtheorem{definition}[theorem]{Definition}

\newtheorem{example}[theorem]{Example}
\newtheorem{proposition}[theorem]{Proposition}

\newtheorem{remark}[theorem]{Remark}

\newcommand{\abs}[1]{\lvert#1\rvert}

\usepackage{color}
\newcommand{\remdc}[1]{\begingroup\color{blue}#1\endgroup}
\newcommand{\remxg}[1]{\begingroup\color{red}#1\endgroup}

\usepackage[all]{xy}

\newcommand{\F}{\mathbb{F}}
\newcommand{\Q}{\mathbb{Q}}

\newcommand{\Z}{\mathbb{Z}}
\newcommand{\CP}{\mathbb{C}P^{\infty}}
\newcommand{\BM}{B_{\textup{M}}}
\newcommand{\EM}{E_{\textup{M}}}
\newcommand{\cc}{\textup{c}}

\newcommand{\Id}{\textup{Id}}
\newcommand{\Ker}{\textup{Ker}}
\newcommand{\col}{\textup{coll}}
\newcommand{\wt}{\widetilde}
\newcommand{\wh}{\widehat}
\newcommand{\xra}{\xrightarrow}

\newcommand{\del}{\partial}
\newcommand{\ol}{\overline}
\newcommand{\ul}{\underline}
\newcommand{\opn}{\operatorname}
\newcommand{\XX}[1]{X^{({#1})}}
\newcommand{\BB}[1]{B^{({#1})}}

\begin{document}

\title{On $H^*(BPU_n; \Z)$ and Weyl group invariants}

\author{Diarmuid Crowley}
\address{School of Mathematics and Statistics, The University of Melbourne, Parkville VIC 3010, Australia}
\curraddr{}
\email{dcrowley@unimelb.edu.au}

\thanks{}

\author{Xing Gu}
\address{Institute for Theoretical Sciences, Westlake University, School of Science, 600 Dunyu Road, Sandun town, Xihu district, Hangzhou 310030, Zhejiang Province, China.} 
\email{guxing@westlake.edu.cn}
\thanks{The second named author is supported by the Young Scientists Fund, NSFC, No.12201503.}

\subjclass[2020]{55R35, 55R40, 55T10}

\date{}

\dedicatory{}

\keywords{}

\begin{abstract}
For the projective unitary group $PU_n$ with a maximal torus $T_{PU_n}$ and Weyl group $W$,
we show that the 
integral restriction homomorphism
\[\rho_{PU_n} \colon H^*(BPU_n;\Z)\rightarrow H^*(BT_{PU_n};\Z)^W\]
to the integral invariants of the Weyl group action is onto.
We also present several rings naturally isomorphic to $H^*(BT_{PU_n};\Z)^W$.

In addition we give general sufficient conditions for the restriction homomorphism $\rho_G$ to be onto for a connected compact Lie group $G$.

\end{abstract}

\maketitle
\section{Introduction}\label{sec:Intro}

Let $G$ be a compact connected Lie group. Its maximal tori are determined up to conjugacy. Let $T_G$ be one of the maximal tori and $N_G$ its normalizer. Then the Weyl group, $W_G\coloneqq N_G/T_G$, acts by conjugation on $T_G$:
we shall often omit the subscript $G$ from the notation.

For a topological group $\Gamma$,
let $B\Gamma$ denote any model for the classifying space of $\Gamma$, together with a choice of universal principal $\Gamma$-bundle over $B\Gamma$, and let $H^*(B\Gamma; R)$ denote cohomology with $R$ coefficients.
The conjugation action of $W$ on $T$ induces an action of $W$ on the cohomology ring 
$H^*(BT;R)$ and the inclusion $T\hookrightarrow G$ induces a ring homomorphism 
%
\[H^*(BG;R)\rightarrow H^*(BT;R),\]
%
whose image lies in $H^*(BT;R)^W$, the subring of invariants under the $W$-action.
Hence we obtain the restriction homomorphism
\begin{equation}\label{eq:Weyl Group Inv}
\rho_{G; R} \colon H^*(BG;R)\rightarrow H^*(BT;R)^W.
\end{equation}

The study of $H^*(BG; R)$ via the homomorphism $\rho_{G; R}$ was
pioneered by Leray \cite{leray1949lanneau}. The modern form of the problem was considered by Borel \cite{borel1953cohomologie}, and later developed by Borel and Hirzebruch \cite{borel1958}. Determining the image of $\rho_{G;R}$ remains an important problem in algebraic topology.
It is well known that $\rho_{G;\,\Q}$ is onto for all $G$: there are numerous proofs,
such as the one in Chapter III of \cite{hsiang2012cohomology}.
It is also well known for
$G = U_n$, the unitary group, that $\rho_{U_n; R}$ is onto for all coefficient groups $R$.
For the projective unitary group $G = PU_n = U_n/S^1$,
where $S^1 \subseteq U_n$ is the center, the situation is less clear:
%
%
Toda proved that $\rho_{PU_n;\,\Z/2}$ is onto for
$n\equiv2\pmod{4}$ and $n=4$ \cite{toda1987cohomology} and we discuss the integral
case below.

In this paper we will focus on integral Weyl group invariants and set
%
\[H^*(-) \coloneqq H^*(-;\Z),\ \ \rho_G \coloneqq  \rho_{G; \Z}.\]
Feshbach \cite{feshbach1981image} proved that $\rho_{SO_n}$ is onto
for all $n$ but that $\rho_{Spin_{12}}$ is not onto.  Later Benson and Wood \cite{benson-wood1995}
proved for all $n \geq 6$ that $\rho_{Spin_n}$ is onto if and only if $n$ is not congruent to $3, 4$ or $5$ modulo $8$.
Kameko and Mimura \cite{kameko2012weyl} studied more cases.
For $G = PU_2$, we have $PU_2 \cong SO_3$ and so $\rho_{PU_2}$ is onto.
Vezzosi showed that $\rho_{PU_3}$ is onto \cite{vezzosi2000chow}
and this was later generalized by Vistoli to $\rho_{PU_p}$ for all odd primes $p$
\cite{vistoli2007cohomology}. The second author also proved for degrees $\ast \leq 12$,
that $\rho_{PU_n}$ is onto for all $n$ \cite{gu2019cohomology}.

Despite the progress discussed above, the question of whether $\rho_{PU_n}$ was onto 
remained opened in general.
Recently, the work of Antieau and Williams \cite{antieau2014topological} and the second author \cite{gu2019topological, gu2020topological} 
showed that
the determination of $H^*(BPU_n)$ is an important input for studying the
topological period\textendash index problem
\cite{antieau2014topological} and this has
renewed interest in the map $\rho_{PU_n}$.  Our main result is the following



\begin{theorem} \label{thm:mainPU}
For all $n$, the restriction homomorphism
\[ \rho_{PU_n} \colon H^*(BPU_n)\rightarrow H^*(BT_{PU_n})^W\]
is onto.
\end{theorem}

We present the outline of the proof of Theorem \ref{thm:mainPU} in Section \ref{ss:outline},
after first discussing some of its consequences. 
Let $q \colon U_n \to PU_n : = U_n/S^1$ 
denote the quotient map defining $PU_n$,
and consider the related fibration sequence of classifying spaces
\begin{equation} \label{eq:BPUn}
BU_n \xra{Bq} BPU_n \to B^2S^1.
\end{equation}
\noindent
As we explain in Section~\ref{s:primitive}, Theorem \ref{thm:mainPU} is
equivalent to the following

\begin{corollary}\label{cor:Bq}
The image of $(Bq)^* \colon H^*(BPU_n) \to H^*(BU_n)$ is a summand, which is
naturally isomorphic to $H^*(BT_{PU_n})^W$.
\end{corollary}

We define four further graded rings naturally isomorphic to $H^*(BT_{PU_n})^W$.


\smallskip
\noindent
(i) For $0 \leq j \leq n$, let $c_j \in H^{2j}(BU_n)$ be $j$th Chern class.
Define the derivation 
\[ \nabla_{\! n} \colon H^*(BU_n) \to H^{*-2}(BU_n)\]
by setting $\nabla_{\! n}(c_j) \coloneqq  (n{-}j{+}1)c_{j-1}$ and extending by the Leibniz rule. Then $\opn{Ker}(\nabla_{\! n})$ is a subring of $H^*(BU_n)$.

\smallskip
\noindent
(ii) Let $\mu \colon BS^1\times BU_n \to BU_n$ be the map which classifies
the homomorphism $S^1 \times U_n \to U_n$, where $S^1 \subset U_n$ is the center
and define the subring of primitive elements
\[ PH^*(BU_n) \coloneqq  \{ c \, \mid \, \mu^*(c) = c \otimes 1 \} \subset H^*(BU_n).\]

\smallskip
\noindent
(iii) Let $d^{0, *}_3 \colon H^*(BU_n) \to H^{*-2}(BU_n)$ denote the $d_3$-differential from the
$0$-column of the Leray\textendash Serre spectral sequence of the fibration \eqref{eq:BPUn}. Since $d^{0, *}_3$ satisfies the Leibniz rule, $\opn{Ker}(d^{0, *}_3)$ is a graded subring of $H^*(BU_n)$.

\smallskip
\noindent
(iv) For a space $X$, we define the quotient ring 
$FH^*(X) \coloneqq  H^*(X)/TH^*(X)$, 
where $TH^*(X)$ is the torsion ideal of $H^*(X)$. The ring to be considered is $FH^*(BPU_n)$.

\begin{theorem} \label{thm:6groups}
There are natural ring isomorphisms
%
\[ FH^*(BPU_n) \! \cong H^*(BT_{PU_n})^W \! \cong \mathrm{Im}((Bq)^*) \! = \opn{Ker}(\nabla_{\!n})
\! = PH^*(BU_n)   \! = \opn{Ker}(d^{0, *}_3). \]
%
\end{theorem}

We briefly discuss 
some consequences of Theorem \ref{thm:6groups}.
The Weyl group invariants $H^*(BT_{PU_n})^W$ are described algebraically by $\opn{Ker}(\nabla_{\!n})$ and topologically by $PH^*(BU_n)$.
The equality $\mathrm{Im}((Bq)^*) = \opn{Ker}(d^{0, *}_3)$ entails that all  differentials $d^{0, *}_r$ for $r\!\,>\!\,3$ vanish in the Leray\textendash Serre spectral sequence of the fibration in \eqref{eq:BPUn}.
In Section~\ref{s:primitive}, the isomorphism $FH^*(BPU_n) \cong H^*(BT_{PU_n})^W$ follows from a split short exact sequence of graded abelian groups 
\begin{equation} \label{eq:A}
0 \to TH^*(BPU_n) \to H^*(BPU_n) \xra{\rho_{PU_n}}  H^*(BT_{PU_n})^W \to 0.
\end{equation}
Vistoli \cite{vistoli2007cohomology} has constructed a splitting of the
above sequence as graded rings in the case that $n$ is an odd prime,
however the situation for general $n$ is open. Hence finding natural splittings of $\rho_{PU_n}$, 
i.e., finding natural constructions of characteristic classes in $H^*(BPU_n)$ which map
onto the Weyl group invariants, remains an interesting open problem; see Remark~\ref{rem:pcc}.

The locations of the proofs of the existence of the isomorphisms Theorem \ref{thm:6groups} are summarised in Diagram~\eqref{eq:All} below. Each isomorphism is labeled with the result most relevant to its proof:

\begin{equation}\label{eq:All}
    \begin{tikzcd}[column sep = 35pt]
    & &  \opn{Im}((Bq)^*) \arrow[d,"\cong", "\mathrm{Prop}\ 4.4"'] \\
    FH^*(BPU_n)  \arrow[r, "\cong","\mathrm{Thm}\ 1.1~"'] &
    H^*(BT_{PU_n})^W \arrow[r,"\cong", "\mathrm{Lem}\ 4.3"'] &
    \opn{Ker}(\nabla_n)  &
    PH^*(BU_n) \arrow[l, "\cong"', "\mathrm{~Prop}\ 4.7"] \\
    & & \opn{Ker}(d^{0, *}_3) \arrow[u, "\cong", "\text{\cite[Cor.\,3.4]{gu2019cohomology}}"']
    \end{tikzcd}
\end{equation}



\subsection{General sufficient conditions for $\rho_G$ to be onto} \label{ss:outline}
We will prove Theorem~\ref{thm:mainPU} by establishing general sufficient
conditions for $\rho_G$ to be onto and verifying that they hold for $G = PU_n$.
Recall that we write $N$ for the normalizer of a maximal torus $T$ of $G$ and consider the inclusion $N \to G$.
As $G$ and $N$ have the same maximal torus
and Weyl group,
the natural map $\nu \colon BN \to BG$
induces an isomorphism $\nu^* \colon H^*(BG; \Q) \to H^*(BN; \Q)$.
Since $G/N$ has Euler characteristic $1$ (see \cite[\S 6]{BECKER19751}),
the composition
\[ H^*(BG) \xra{\nu^*} H^*(BN) \xra{\tau^*}H^*(BG) \]
is the identity, where $\tau^*$ denotes
the transfer in cohomology \cite[Theorem 5.5]{BECKER19751}.
Hence $\nu^* \colon H^*(BG) \to H^*(BN)$ is a split injection; see, for example,
\cite[Proposition A.1]{curtis-wiederhold-williams1974} and its proof.
It follows that the induced map
$$ FH^*(BG) \to FH^*(BN)$$
is an isomorphism. 
Since $H^*(BT)$ is torsion--free, so is $H^*(BT)^W$ and 
it follows that the natural
map $H^*(BN) \to H^*(BT)^W$ factors through the induced map $FH^*(BN) \to H^*(BT)^W$.
Thus we have proven

\begin{lemma} \label{lem:BG_to_BN}
The restriction homomorphism $\rho_G \colon H^*(BG) \to H^*(BT)^W$ is onto if and only
if
$\rho_N \colon H^*(BN) \to H^*(BT)^W$ is onto. \qed
\end{lemma}

By definition, there is a short exact sequence $1 \to T \to N \xra{\psi} W \to 1$,
and to investigate when $\rho_N$ is onto,
we consider a more general situation,
where there is a short exact sequence of topological groups,
\begin{equation} \label{eq:AMV}
1\to A\to M\xrightarrow{\varphi} V\to 1,
\end{equation}
with $A \subseteq M$ a closed normal subgroup and $V$ a finite discrete group. We further assume that each of the above topological groups  with its unit forms an NDR pair (6.2 of \cite{steenrod1967convenient}). This condition is satisfied when $A$ is a maximal torus of a compact Lie group and $M$ is the normalizer of $A$.
%
%
Up until now, we have used that the assignment $G \mapsto H^*(BG;R)$ defines a contravariant functor from the category of Lie groups to the category of $R$-algebras. By fixing a functorial construction of $BG$, this functor can be factored through the 
category of 
topological spaces. However, we will need to consider more general models for classifying spaces
and so we introduce the following definitions and related terminology.

\begin{definition}\label{def:model}
Let $\Gamma$ be a topological group. A {\em model for $B\Gamma$} is a universal principal $\Gamma$-bundle $p_X \colon \tilde{X}\to X$, and if $X$ is a $CW$-complex then we call $p_X$ a {\em $CW$-model for $B\Gamma$}.
If $p_X \colon \tilde{X}\to X$ and $p_Y \colon \tilde{Y}\to Y$ are respectively models for the classifying spaces $B\Gamma_1$ and $B\Gamma_2$ of topological groups $\Gamma_1$ and $\Gamma_2$, 
and if $\varphi \colon \Gamma_1 \to \Gamma_2$ is a continuous homomorphism, then a map $f \colon X \to Y$ is a 
{\em model for $B\varphi\colon B\Gamma_1 \to B\Gamma_2$} if the pull-back principal $\Gamma_2$-bundle $f^*p_Y$ is isomorphic to the principal $\Gamma_2$-bundle $\tilde{X}\times_{\Gamma_1}\Gamma_2\to X$. 
\end{definition}

The extension $A \to M \xra{\varphi} V$ in~\eqref{eq:AMV} ensures there are models for $BA$ with $V$-actions (see Section~\ref{s:LHSSS}),
and fixing a model for $BV$, we obtain
the homotopy orbit space $BA/\!/V \coloneqq EV \times_V BA$,
which comes equipped with a projection map $BA/\!/V\to BV$.
In many cases $BA/\!/V$ is a model for $BM$, which leads us to

\begin{definition} \label{def:adapted}
A model $p \colon EA \to BA$ for $BA$, where $V$ acts on $BA$
is called {\em $\varphi$-adapted} if the following conditions hold:
\begin{enumerate}
    \item There is a universal principal $M$-bundle  $EM\to BA/\!/V$, which is a model for $BM$;
    \item The natural projection $BA/\!/V\to BV$ is a model for $B\varphi \colon BM \to BV$.
\end{enumerate}
\end{definition}

In Section~\ref{s:LHSSS} we recall 
arguments showing that 
the map classifying $A \to M$ has a model $\pi \colon BA \to BM$ in which $\pi$ is a $V$-covering space and $BM$ is a $CW$-complex. 
The $CW$-structure on $BM$ can then be lifted to $BA$ 
(see Lemma~\ref{lem:CellLifting}), 
making $BA$ a $V$-$CW$-complex, 
i.e., a $CW$-complex with a cellular $V$-action, where a set-wise fixed cell is also point-wise fixed.  
As we will encounter cellular $V$-actions which do not define $V$-$CW$-complexes, we make the following

\begin{definition}\label{def:weakCW}
A {\em weak $V$-$CW$-complex} is a $CW$-complex with a cellular $V$-action.
\end{definition}
%

\noindent
In Lemma~\ref{lem:BM}, we prove that $\varphi$-adapted $V$-$CW$-models 
for $BA$ (\emph{a fortiori} weak $V$-$CW$-models for $BA$) always exist.

\begin{theorem} \label{thm:main}
Let $1\to A\to M\xrightarrow{\varphi}V\to 1$
%
%
be a short exact sequence of topological groups, where 
$V$ is finite and discrete. Then the following are equivalent:
\begin{enumerate}
\item\label{ite:1}
The natural map $H^*(BM; R) \to H^*(BA; R)^V$ is onto;
\item\label{ite:2}
For some $\varphi$-adapted weak $V$-$CW$-model of $BA$,
the natural map of $V$-invariants
\[Z^*(BA; R)^V\rightarrow H^*(BA; R)^V\]
is onto.
\end{enumerate}
\end{theorem}


As a direct consequence of Lemma \ref{lem:BG_to_BN} and Theorem \ref{thm:main} we have

\begin{corollary} \label{cor:main}
Let $G$ be a connected compact Lie group with $N$ the normalizer of the maximal torus $T$ and $\psi \colon N \to W = N/T$ the
surjection to the Weyl group. 
Then the 
restriction homomorphism
$\rho_G \colon H^*(BG) \to H^*(BT)^W$ is onto if and only if there is a $\psi$-adapted weak $W$-$CW$-model for $BT$
such that the induced map of $W$-invariants
\[Z^*(BT)^{W}\rightarrow H^*(BT)^{W}\]
is onto. \qed
\end{corollary}

\noindent
We will prove Theorem \ref{thm:mainPU} by showing that $BT_{PU_n}$ has a weak $W$-$CW$-model
satisfying the condition of Corollary \ref{cor:main}; see Section \ref{s:CWdecomp}.

%


As mentioned earlier, Benson and Wood \cite{benson-wood1995} proved that Condition \eqref{ite:1} of Theorem \ref{thm:main} fails for $G = Spin_n$, when $n\!>\!6$ and $n \equiv 3,4,5\pmod{8}$. Therefore, Theorem \ref{thm:main} imposes restrictions on possible $W$-equivariant $CW$-decompositions of the classifying space of the maximal torus of $G$. More precisely, we have the following

\begin{corollary}\label{cor:BSpin}
Let $G = Spin_n$ and $\psi \colon N \to W$ be the projection to
the Weyl group.
For $n\!>\!6$ and $n \equiv 3,4,5\pmod{8}$,
there is no $\psi$-adapted weak $W$-$CW$-complex model of $BT$
such that the homomorphism of $W$-invariants
\[Z^*(BT)^W\to H^*(BT)^W\]
is onto. \qed
\end{corollary}

Tits \cite{tits1966} and independently Curtis, Wiederhold and Williams \cite{curtis-wiederhold-williams1974}, showed that for $G = Spin_n$, $n > 2$, the canonical surjection $N \to W$ does not split, 
and the results of Benson and Wood \cite{benson-wood1995}
show that in these cases $\rho_{Spin_n}$ is sometimes onto and sometimes not onto.
At the time of writing, there is no known 
$G$ for which $N \to W$ splits but $\rho_G$ is not onto.
Hence we ask the following

\begin{question}\label{q:rho_G_onto}
If $G$ is a connected compact Lie Group and $N \cong T \rtimes W$, is $\rho_G$ onto? 
Equivalently, does $BT$ admit a $\psi$-adapted weak $W$-$CW$-structure
where $Z^*(BT)^W \to H^*(BT)^W$ is onto?
\end{question}

When $G$ is not a connected compact Lie group, 
there are examples of
Benson and Feshbach \cite{benson1994cohomology}, where the short exact sequence
$1\to A \to M \xrightarrow{\varphi} V \to 1$
splits, but the conclusion of Theorem~\ref{thm:main} is not satisfied.
Specifically, for each positive integer $n$, Benson and Feshbach 
constructed a finite discrete group $M_n$ with a split extension
\[1 \to A_n \to M_n \xra{\varphi} \Z/2 \to 1,\]
where $A_n$ is a non-abelian group, and where 
the differential $d_n^{0,n-1}$ is nontrivial
in the Lyndon--Hochschild--Serre spectral sequence
\begin{equation*}
    \begin{split}
        & E_2^{s,t}=H^s(B\Z/2;H^t(BA_n;\F_2)_\phi)\Rightarrow H^{s+t}(BM_n;\F_2).\\
    \end{split}
\end{equation*}
Here the local coefficient system $H^t(BA_n;\F_2)_\phi$ is twisted by the action $\phi$ of the fundamental group $\pi_1(B\Z/2) = \Z/2$ on 
$H^*(BA_n; \F_2)$,
and it follows that the
restriction homomorphism 
$H^*(BM_n; \F_2) \to H^*(BA_n; \F_2)^{\Z/2}$
factors through $\Ker(d_n^{0, n-1})$.
Since $d_n^{0,n-1}$ is nontrivial, the restriction homomorphism 
is not surjective, and so by Theorem \ref{thm:main}, we have

\begin{corollary}\label{cor:BF}
For the Benson--Feshbach groups
$M_n  \cong A_n \rtimes \Z/2$,
there is no $\varphi$-adapted weak 
$(\Z/2)$-$CW$-complex structure on $BA_n$
such that
the canonical homomorphism
\[Z^*(BA_n; \F_2)^{\Z/2}\to H^*(BA_n; \F_2)^{\Z/2}\]
is onto. \qed
\end{corollary}

As a final comment, we point out that determining the cohomology of $PU_n$ is already a delicate problem. When $p$ is a prime, the mod $p$ cohomology ring $H^*(PU_n; \Z/p)$ was determined by Baum and Browder \cite{baum-browder1965}.
The integral cohomology ring $H^*(PU_n)$ was completely determined by Duan \cite{duan2017},
and the complexity of $H^*(PU_n)$ is one of the reasons that 
the computation of $H^*(BPU_n)$ remains a difficult problem in general.

\subsection*{Organisation}
In Section \ref{s:LHSSS} we prove Theorem \ref{thm:main} by studying the Serre spectral sequence of the fibration $BA \to BM \to BV$ and related spectral sequences.
In Section \ref{s:CWdecomp} we prove there are
$\psi$-adapted weak $W$-$CW$-complex models for
$BT_{PU_n}$, proving Theorem \ref{thm:mainPU}.
Section \ref{s:primitive} discusses the Weyl group invariants $H^*(T_{PU_n})^W$,
proving Corollary~\ref{cor:Bq} and Theorem~\ref{thm:6groups}.
We conclude with appendicies on $CW$-structures on the total spaces of covering spaces and on Milgrams construction of the classifying space of a topological group.

\subsection*{Acknowledgements}
We would like to thank Dave Benson and Adrian Hendrawan for helpful comments on
earlier versions of this paper,
and also the anonymous referees for numerous helpful comments, which have inspired significant revisions.
The second author would like to thank the Max Planck Institute for Mathematics in Bonn for its hospitality and support.

\section{Homotopy orbit firbrations 
for weak $V$-$CW$-complexes} \label{s:LHSSS}
%
%
%
%
%
In this section we prove Theorem \ref{thm:main}.
Recall that $\varphi \colon M \to V$
is a continuous surjective homomorphism of topological groups, where $V$ is finite and
discrete, and that $A \subseteq M$ is the kernel of $\varphi$.
Recall also that if $p \colon EA \to BA$ is a model for $BA$, 
then $p$ is called $\varphi$-adapted if there is a $V$-action on $BA$ with a universal principal $M$-bundle on $BA/\!/V$ such that the natural map $BA/\!/V \to BV$ is a model for $B\varphi \colon BM \to BV$.
%
%
%





\begin{lemma} \label{lem:BM}
For all $\varphi \colon M \to V$,
$\varphi$-adapted $V$-$CW$-models for $BA$ exist.
%
%
\end{lemma}

\begin{proof}
Since $V$ is finite and discrete,
there is a model $EV \to BV$ for $BV$, 
which is a $V$-covering space.
Since classifying spaces are well defined only up to weak homotopy equivalence, 
$BM$ has a model which is a $CW$-complex.
Suppose that $f \colon BM \to BV$ is a map such that $f_* \colon \pi_1(BM) \to \pi_1(BV) = V$
is equal to the induced map $\varphi_* \colon \pi_0(M) \to \pi_0(V) = V$:
such a map $f$ exists since $BV$ is a $K(V, 1)$.
Now let $X \to BM$ be the pullback of the $V$-covering space $EV \to BV$.
By Lemma \ref{lem:Gcell}, the $CW$-structure on $BM$ pulls back to define a $V$-$CW$-structure on $X$.
We have $BM = X/V$,
and in the commutative diagram
\begin{equation*}
    \begin{tikzcd}
        X/\!/V \arrow[rr] \arrow[rd] & & X/V \arrow[ld, "f"] \\
              & BV &,
    \end{tikzcd}
\end{equation*}
the canonical map $X/\!/V \to X/V = BM$ is a weak homotopy equivalence, since the $V$-action of $X$ is free.  It follows that $X/\!/V \to BV$ is a model for $B\varphi \colon BM \to BV$, and so $X$ is a $\varphi$-adapted $V$-$CW$-model for $BA$.
\end{proof}

\begin{proof}[Proof of Theorem \ref{thm:main}]
Throughout this proof, we suppress the coefficient ring $R$.

We first show that Condition (1) implies Condition (2).  By the proof of Lemma~\ref{lem:BM}, there is a $\varphi$-adapted $V$-$CW$-model for $BA$ with $BM = BA/V$.  In this case the quotient map
$\pi \colon BA \to BM = BA/V$ is such that
\[ \pi^{\#} \colon Z^*(BM) \to Z^*(BA), \]
%
takes values in $Z^*(BA)^V$. Consider the commutative diagram
\begin{equation}\label{eq:ZBM_HBM_diagram}
    \begin{tikzcd}
        Z^*(BA)\arrow[d, two heads]\arrow[r, "\pi^{\#}"]& Z^*(BA)^V\arrow[d]\\
        H^*(BM)\arrow[r, "\pi^*"]& H^*(BA)^V
    \end{tikzcd}    
\end{equation}
where the vertical arrow on the left is the quotient homomorphism, which is onto. In addition, Condition (1) states that  the horizontal arrow on the bottom
\[ \pi^* \colon H^*(BM) \to H^*(BA)^V\]
is onto. Therefore, so is the vertical arrow on the right. 

To show that Condition (2) implies Condition (1), 
we suppose that $X$ is a $\varphi$-adapted weak $V$-$CW$-model
for $BA$ such that $Z^*(X)^V \to H^*(X)^V$ is onto, 
and consider the 
Serre spectral sequence associated to 
the homotopy fibre sequence $X \to X/\!\!/V \to BV$,
with $E_2$~page $E_2^{s,t}\cong H^s(BV;H^t(X)_{\phi})$ and converging to $H^*(X/\!\!/V)$. Here $H^t(X)_{\phi}$ is the local coefficient system given by the canonical action $\phi$ of $V = \pi_1(BV)$ on $H^*(X)$.
%

On the $E_2$-page we have the leftmost column
$E^{0, k}_2 \cong H^k(X)^V$ and on the $E_\infty$-page
we have that $E^{0, k}_\infty = 
\mathrm{Im}(H^k(X/\!\!/V) \to H^k(X)^V)$.
Since the $V$-action on $X$ is $\varphi$-adapted,
we may take $BM = X/\!\!/V$, 
and so it suffices to show that all differentials coming out of the $0$th column are trivial.

For $k = 0$, all differentials leaving $E^{0, 0}_*$ vanish for dimensional reasons.
To understand the differentials for $k \geq 1$, 
we note that $V$ acts on $X^{(k-1)}$, the $(k{-}1)$-skeleton of $X$, and so on the quotient $CW$-complex 
$X_{k} \coloneqq X/X^{(k-1)}$.
It follows that $X_k$ is weak $V$-$CW$-complex and 
that the collapse map $\col_k \colon X \to X_k$ is cellular and $V$-equivariant.
Passing to homotopy orbits, we obtain 
the $V$-equivariant collapse map 
%
\[\col_k/\!/V \colon EV\times_V X\rightarrow EV\times_V X_{k},
\quad [e, x] \mapsto [e, \col_k(x)],
\]
which fits into the following commutative diagram of fibrations over $BV$:
%

\begin{equation}\label{eq:X fiberseq}
    \begin{tikzcd}
        X \arrow[r] \arrow[d,"\col_k"] &
    EV \times_V X\arrow[r, "\pi"] \arrow[d,"\col_k/\!/V"] &
BV\arrow[d, "\Id"] \\
X_{k}\arrow[r] & EV \times_V X_{k} \arrow[r,"\pi_{k}"] &
BV.
    \end{tikzcd}
\end{equation}

Moreover, for $k \geq 1$, the $V$-action on $X_{k}$ has a fixed point and so $\pi_{k}$ has a section.

Let $_{k}E_*^{*,*}$ be the Serre spectral sequence associated to the homotopy fiber sequence
\begin{equation}\label{eq:fiberprod}
X_{k}\rightarrow EV \times_V X_{k}\rightarrow BV.
\end{equation}
The commutative diagram \eqref{eq:X fiberseq} induces a morphism of spectral sequences
\[ E^{*,*}(\col_k/\!/V) \colon _{k}E_*^{*,*}\rightarrow E_*^{*,*},\]
which on the $E_2$-page is given by the homomorphism
\[ E^{s, t}(\col_k/\!/V) \colon H^s(BV;H^t(X_{k})_{\phi}) \rightarrow H^s(BV; H^t(X)_{\phi}). \]
%
In particular, for $(s, t) = (0, k)$, we have
\begin{equation}\label{eq:surj}
E^{0, k}(\col_k/\!/V)=(\col_k/\!/V)^* \colon _{k}E_2^{0,k}=H^k(X_k)^V=Z^k(X) \rightarrow H^k(X)^V=E_2^{0,k}
\end{equation}
which is part of the following commutative diagram:

\begin{equation}\label{eq:HBM_HBV_diagram}
    \begin{tikzcd}
        _kE_{\infty}^{0,k}\arrow[r]\arrow[d, " E^{*,*}(\col_k/\!/V)"']&_kE_2^{0,k}=Z^k(X)^V\arrow[d, two heads, " E^{*,*}(\col_k/\!/V)"]\\
        E_{\infty}^{0,k}\arrow[r]& E_2^{0,k}=H^k(X)^V.
    \end{tikzcd}    
\end{equation}
By Condition (2) of Theorem \ref{thm:main}, the vertical arrow on the left of \eqref{eq:HBM_HBV_diagram} is a surjection. By construction, $X_{k}$ is $(k{-}1)$-connected. Hence in the spectral sequence $_{k}E_*^{*,*}$ we have
\[_{k}E_{k}^{0,k}={_{k}E}_2^{0,k}.\]
On the other hand, the homotopy fiber sequence (\ref{eq:fiberprod}) has a section,
and therefore the differential
\[_{k+1}d_k^{0,k}:{_{k}E}_{2}^{0,k}={_{k}E}_{k}^{0,k}\rightarrow {_{k}E}_{k}^{k+1,0}\]
is trivial and the top horizontal arrow of \eqref{eq:HBM_HBV_diagram} is the identity. Therefore, the bottom horizontal arrow is surjective, which completes the 
proof that Condition (2) implies Condition (1).
\end{proof}

%

\begin{remark}
We conclude this section with a discussion of Condition (2) of Theorem~\ref{thm:main},
which requires that $BA$ has a $\varphi$-adapted weak 
$V$-$CW$ model $X$ in which the natural map 
$Z^*(X; R)^V \to H^*(X; R)^V$ is onto.
By definition, there is a short exact sequence of $RV$-modules
\begin{equation} \label{eq:coefficients}
 0 \to B^i(X; R) \to Z^i(X; R) \to H^i(X; R) \to 0,
\end{equation}
where $B^i(X; R) = \partial(C^{i-1}(X; R)) \subseteq Z^i(X; R)$
is the submodule of coboundaries.
Writing $B^i_X \coloneqq B^i(X; R), Z^i_X \coloneqq Z^i(X; R)$ and $H^i_X \coloneqq H^i(X; R)$
and regarding \eqref{eq:coefficients} as a short exact sequence of coefficient $RV$-modules, we have the
the exact sequence
\[ H^0(V; Z^i_X) \to H^0(V; H^i_X) \xra{\delta_X} H^1(V; B^i_X) \xra{i_X} H^1(V; Z^i_X). \]
Using the identifications $(Z^i_X)^V = H^0(V; Z^i_X)$ and $(H^i_X)^V = H^0(V; H^i_X)$ we see
the second condition of Theorem \ref{thm:main}
holds whenever there is $\varphi$-adapted 
weak $V$-CW model $X$ for $BA$ where $\delta_X = 0$, or
equivalently where $i_X$ is injective.
\end{remark}
%

\section{A weak $W$-$CW$-model for $BT_{PU_n}$} \label{s:CWdecomp}
The starting point of this section is the split short exact sequence of topological groups 
\[1\to T_{PU_n}\to N_{PU_n}\xrightarrow{\psi}W\to 1,\]
where $T_{PU_n}$ is a maximal subgroup of $PU_n$, $N_{PU_n}$ is the normalizer of $T_{PU_n}$ in $PU_n$, and $W$ is the Weyl group. We recall this construction as follows.

Fix a maximal torus $T_{U_n} = T^n \subset U_n$ and let 
$T_{PU_n} = T^{n-1} = T^n/S^1$ be the corresponding maximal torus of $PU_n$, where $S^1 \subset T^n = (S^1)^{\times n}$ is the diagonal. 
The Weyl group of $U_n$, $W = S_n$, acts by conjugation on $T^n$, permuting factors and the Weyl group of $PU_n$ acts on $T^{n-1}$ via the induced action, since $S^1$ is fixed by $W$.  Moreover, both normalizers $N_{U_n}$ and $N_{PU_n}$ are semi-direct products, $N_{U_n} = T^n \rtimes W$ and $N_{PU_n} = T^{n-1} \rtimes W$, and there is a commutative diagram of homomorphisms of topological groups
\begin{equation*}
    \begin{tikzcd}
       T^n = T_{U_n} \arrow[d,"q"] \arrow[r] &
        N_{U_n} \arrow[d] \arrow[r] & 
        W \arrow[d,"="] \\
       T^{n-1} = T_{PU_n} \arrow[r] &
        N_{PU_n} \arrow[r, "\psi"] &W.
    \end{tikzcd}
\end{equation*}
\noindent
Here, the rows are split exact, and 
the quotient homomorphism $q \colon T^n \to T^{n-1}$ is $W$-equivariant.
We have abused notation, using $q$ to also denote  
the restriction of $q \colon U_n \to PU_n$ to $T^n \subset U_n$.
In particular, $W$ acts on $T^{n}$ and $T^{n-1}$ via group automorphisms such that the homomorphism $q \colon T^n \to T^{n-1}$ is $W$-equivariant.

In this section, we construct a weak $W$-$CW$-complex model for $BT^{n-1}$ satisfying
Condition~(2) of Theorem \ref{thm:main};
i.e., we will prove
%

\begin{proposition}\label{pro:$CW$-decomp}
There is a $\psi$-adapted weak $W$-$CW$-model for $BT_{PU_n}$ such that
the induced homomorphism of $W$-invariants
\[Z^*(BT_{PU_n})^W\rightarrow H^*(BT_{PU_n})^W\]
%
is onto.  
%
\end{proposition}



%

\begin{remark}
The analogue of Proposition~\ref{pro:$CW$-decomp} is presumably well-known for $U_n$, and this is one of the starting points of our 
proof.  Indeed, as we discuss below, if $P$ denotes the standard $CW$-decomposition of $\CP$, $P^{\times n}$ is given the product $CW$-structure and $W = S_n$ acts on
$P^{\times n}$ by permuting co-ordinates, then $P^{\times n}$ is a weak $W$-$CW$-model for $BT_{U_n}$ 
where $Z^*(P^{\times n})^W \to H^*(P^{\times n})^W$ is an isomorphism.
\end{remark}

To prove Proposition~\ref{pro:$CW$-decomp} we will use Milgram's construction, $\EM \Gamma \to \BM \Gamma$, of the 
universal bundle of a topological group $\Gamma$. 
In Appendix \ref{sec:Milgram} 
we review the essential facts about Milgram's construction which we will need. By Proposition \ref{pro:BMfunctor}, the assignments $\Gamma \mapsto \EM \Gamma$ and $\Gamma \mapsto \BM \Gamma$ define functors from the category of topological groups to the category of topological spaces. 
Since $q\colon T^n\to T^{n-1}$ is $W$-equivariant and $B_M$ is a functor, it follows that the map  
$\BM q \colon \BM T_{U_n}\to \BM T_{PU_n}$ is $W$-equivariant.


By Proposition~\ref{pro:phi-adapted}, the $W$-action on $B_M T^{n-1}$ is $\psi$-adapted.
Turning to the $W$-action on $B_M T^n$,
since $T^n = (S^1)^{\times n}$,
Lemma \ref{pro:S_n_equiv} gives a $W$-equivariant homeomorphism $\zeta \colon (\BM S^1)^{\times n} \to \BM T^n$, where $W$ acts on $(\BM S^1)^{\times n}$ by permuting the coordinates.
Now $B_M S^1$ is a $K(\Z,2)$, and we let $P$ be a $CW$-complex of the same homotopy type with no odd-dimensional cells and one cell in each even dimension (e.g.\ we can take $P$ to be the standard $CW$-decomposition of $\CP$) and we let $\mu \colon P\to B_M T^1$ be a homotopy equivalence.
Summarizing the above discussion,
we have the following 

\begin{lemma}\label{lem:Bq-input}
Let $W = S_n$ act on $P^{\times n}$ by permuting the coordinates and on $\BM T^{n-1}$ as the Weyl group of $PU_n$. Then
\begin{enumerate}
\item The $W$-action on $\BM T^{n-1}$ is $\psi$-adapted;
\item
The composition
\[f \coloneqq (\BM q) \circ \zeta \circ (\mu^{\times n})
\colon P^{\times n}\to \BM T^{n-1}\]
is a $W$-equivariant map, which is a model for
$Bq \colon BT^n \to BT^{n-1}$. \qed
\end{enumerate}

\end{lemma}

The $W$-equivariant map
$f \colon P^{\times n} \to \BM T^{n-1}$,
is the starting point for our 
proof of Proposition~\ref{pro:$CW$-decomp}.
We will attach $W$-equivariant cells to $P^{\times n}$, using $f$ as a guide to build the desired $W$-$CW$-model for $BPU_n$.  To verify that the outcome of this process is fit for our purpose, and to carry it out in the first place, we require three preliminary lemmas.
The first of these is stated in the more general setting of Section~\ref{s:LHSSS}.  It uses the notation from that section
%
and is a simple observation, which can nonetheless be useful for determining when a weak $V$-$CW$ complex model for $BA$ is $\varphi$-adapted.

\begin{lemma} \label{lem:easy_obs}
Suppose that $X$ and $Y$ are models for $BA$ with $V$-actions
and that $f \colon X \to Y$ is a $V$-equivariant map 
which is a weak homotopy equivalence.
If the $V$-action on $Y$ is $\varphi$-adapted,
then so is the $V$-action on $X$.
%
\end{lemma}

\begin{proof}
Since $f$ is a $V$-equivariant map, it induces a map of homotopy quotients
$f/\!/V \colon X/\!/V \to Y/\!/V$, which fits into the following map of fibrations
\[
\xymatrix{
X \ar[d] \ar[r]^(0.5)f &
Y \ar[d] \\
X/\!/V \ar[d] \ar[r]^(0.5){f/\!/V} &
Y/\!/V \ar[d] \\
BV \ar[r]^= &
BV.}
\]
The 5-Lemma shows that $(f/\!/V)_* \colon \pi_*(X/\!/V) \to \pi_*(Y/\!/V)$ is an isomorphism. Therefore, $f/\!/V$ is a weak homotopy equivalence which commutes with the maps to $BV$, and the lemma follows.
\end{proof}

We next establish two elementary lemmas about cell attachments and extensions of maps over them.
Henceforth, we assume that all spaces and maps are pointed and that the domains of all maps are $CW$-complexes, where the base-point lies in the $0$-skeleton.
Let $Y$ be a $CW$-complex and $f \colon Y \to Z$ a map.
For simplicity, we assume that $Y$ and $Z$ are simply-connected. We use $f$ to regard $(Z, Y)$ as a pair. For $k \geq 2$, suppose there is a finitely generated free abelian group $F$ and a  homomorphism $p\colon F\to \pi_{k+1}(Z,Y)$. Let $\ul \gamma' = (\gamma_1', \dots,  \gamma_t')$ and 
$\ul \delta' = (\delta_1', \dots, \delta_t')$ be two bases of $F$ and set $\gamma_i := p(\gamma_i')$ and $\delta_i := p(\delta_i')$. 

We define the $CW$-complexes
\[ \ol Y_{\ul \gamma} := Y \cup \left( \cup_{i=1}^{t} e^{k+1}_i \right) 
\quad \text{and} \quad
\ol Y_{\ul \delta} := Y \cup \left( \cup_{i=1}^{t} e^{k+1}_i \right), 
\]
to be the result of attaching $(k{+}1)$-cells to $Y$ along the maps $\gamma_1|_{S^k}, \dots, \gamma_t|_{S^k}$,
resp.\ $\delta_1|_{S^k}, \dots, \delta_t|_{S^k}$, 
and we let $\wh f_{\ul \gamma} \colon Y_{\ul \gamma} \to Z$ and $\wh f_{\ul \delta} \colon Y_{\ul \delta} \to Z$ be the extensions of $f$ defined by $\ul\gamma = (\gamma_1, \dots, \gamma_t)$,
resp.\ $\ul \delta = (\delta_1, \dots, \delta_t)$. 

\begin{lemma} \label{l:cells}
In the situation above, there is a cellular homotopy equivalence of pairs
$g \colon (\ol Y_{\ul \gamma}, Y) \to (\ol Y_{\ul \delta}, Y)$ such that $g|_Y = \Id_Y$.  In addition, if $\pi_{k+1}(Z) = 0$, then
$\wh f_{\ul \gamma}$ is homotopic rel.\ $Y$ to $\wh f_{\ul \delta} \circ g$.
\end{lemma}

\begin{proof}
Let $A = (a_{ij})\in GL_t(\Z)$ be the change of basis matrix defined from $\ul \gamma'$ to $\ul \delta'$, i.e., $\ul \gamma' = \ul \delta'A$. 

There is a map $g_A \colon \vee_{i=1}^t S^k \to \vee_{i=1}^t S^k$ realizing $A$ on $\pi_k(\vee_{i=1}^t S^k) = \Z^t$, i.e., if $\iota_j\in\pi_k(\vee_{i=1}^t S^k)$ denotes the class represented by the inclusion of the $j$th copy of $S^k$, then 
\[(g_A)_*(\iota_j) = \sum_{i=1}^t a_{ij}\iota_i,\]
which implies 
\[\vee_{i=1}^t\gamma_i\simeq (\vee_{i=1}^t\delta_i)\cdot g_A\colon \vee_{i=1}^t S^k\to Y.\]
Therefore, we have a homotopy commutative diagram
\begin{equation*}
    \begin{tikzcd}
        \vee_{i=1}^t S^k\arrow[r, "\vee_{i=1}^t\gamma_i"]\arrow[d, "g_A"]& Y\arrow[r]\arrow[d,"="]& \ol Y_{\ul \gamma}\arrow[d, dashed, "g"]\\
        \vee_{i=1}^t S^k\arrow[r, "\vee_{i=1}^t\delta_i"]& Y\arrow[r]&\ol Y_{\ul \delta}.
    \end{tikzcd}
\end{equation*}
The dashed arrow $g$ exists because both rows are cofiber sequences. It is a weak equivalence because the other two vertical arrows are. By construction, $g \colon \ol Y_{\ul \gamma} \to \ol Y_{\ul \delta}$ is a cellular map. Therefore it is a homotopy equivalence by Whitehead's Theorem.

Finally, since 
$\wh f_{\ul \gamma}|_Y = (\wh f_{\ul \delta} \circ g)|_Y$, the obstruction to the existence of a homotopy rel.\ $Y$ between $\wh f_{\ul \gamma}$ and $\wh f_{\ul \delta} \circ g$ is a cocycle in $C^{k+1}(\ol Y_{\ul \gamma}; \pi_{k+1}(Z)) = 0$.  
By obstruction theory \cite[ Ch.\ 5, Theorem 5.15]{whitehead1978},
$\wh f_{\ul \gamma}$ is homotopic rel.\ $Y$ to $\wh f_{\ul \delta} \circ g$.
%
\end{proof}

Now let $V$ be a finite group, $Y$ a weak $V$-$CW$-complex  having only even-dimensional cells and a $V$-fixed $0$-cell, $Z$ a $V$-space and $f \colon Y \to Z$ a $V$-equivariant map.  
We suppose that $M \subseteq \pi_{k+1}(Z, Y)$ is a $\Z V$-module satisfying the following conditions:
\begin{enumerate}[label=\textbf{M\arabic*}]
    \item\label{ite:M1} As an abelian group, $M$ is free and finitely generated.
    \item\label{ite:M3} Let $\rho$ denote the Hurewicz homomorphism. The composition
    \[M\subseteq\pi_{k+1}(Z,Y)\xrightarrow{\del}\pi_k(Y)\xrightarrow{\rho}H_k(Y)\]
    is injective.
\end{enumerate}

Suppose there is a free $\Z V$-module $F$ with a basis $\gamma_1'',\cdots,\gamma_s''$ and a surjective $\Z V$-module homomorphism $p\colon F \to M$. The underlying abelian group of $F$ is free and of rank $t = s\abs{V}$. Let $\gamma_1',\cdots, \gamma_t'$ be the sequence formed by $v\gamma_i''$ for all $v\in V$ and $1\leq i\leq s$, which is a basis for the underlying free abelian group of $F$. Let $p(\gamma_i') = \gamma_i$. 
We define the weak $V$-$CW$-complex
\[ \wh Y \coloneqq  \ol Y_{\ul\gamma}, \]
and let $\wh f \colon \wh Y \to Z$ be the $V$-equivariant extension of $f$ defined by $\gamma_1, \dots, \gamma_t$.  
The group $\pi_{k+1}(\wh Y, Y)$ is identified with $F$, 
and we let $K$ denote the kernel of $p\colon F\to M$, which is a $\Z V$--module with an underlying free abelian group.  We denote its rank by $l$.

Let $\sigma \colon M \to F$ be a splitting of $p$ in the category of abelian groups and let 
$\delta_1, \dots, \delta_n \in M$ be a basis for the free abelian group $M$.  We define the $CW$-complex
\[ \ol Y : = \ol Y_{\ul\delta}\]
and $\bar f \colon \ol Y \to Z$ the extension of $f$ defined by $\delta_1, \dots, \delta_n$.

\begin{lemma} \label{l:W-cells}
In the situation above, if $\pi_{k+1}(Z) = 0$, then there is a (possibly non-equivariant) cellular homotopy equivalence 
$g \colon \wh Y \to \ol Y \vee \left( \vee_{i = 1}^l S^{k+1} \right)$ with $l = t - n$, and a map 
$f' \colon \vee_{i = 1}^l S^{k+1} \to Z$ 
 such that
\begin{enumerate}
\item The map $\wh f$ is homotopic to $(\bar f \vee f') \circ g$;
\item The inclusion $i_Y \colon Y \to \ol Y$ induces an isomorphism
$i_{Y*} \colon H_{k+1}(Y) \to H_{k+1}(\ol Y)$;
\item The composition 
\[ H_{k+1}(\wh Y) \xra{g_*} H_{k+1}(\ol Y \vee \left( \vee_{i= 1}^l S^{k+1} \right))
\xra{=} H_{k+1}(\ol Y) \oplus K
\xra{(i_Y)^{-1}_* \oplus \Id_K} H_{k+1}(Y) \oplus K\]
is an isomorphism of $\Z V$-modules;
%
\item Under the isomorphism above, $K$ is in the image of the Hurewicz homomorphism $\rho\colon \pi_{k+1}(\wh Y)\to H_{k+1}(\wh Y)$.
\end{enumerate}
\end{lemma}

\begin{proof}
(1) Set $\delta_i' := \sigma(\delta_i)$ for $1\leq i\leq n$, 
and let $\delta'_{n+1},\cdots,\delta'_t$ be a basis of the underlying abelian group of $K$. Then 
$\delta_1',\cdots, \delta_t'$ is a basis of the underlying abelian group of $F$ . 
For $n+1\leq i\leq t$, set $\delta_i := p(\delta'_i) = 0$. 

Order the elements of $V$ in a list $e = (g_1, \dots, g_{|V|})$.
We apply Lemma~\ref{l:cells} to 
$\ul \gamma = (\gamma_1, \dots,  \gamma_t)$ and 
$\ul \delta = (\delta_1, \dots, \delta_n, \delta_{n+1} \dots \delta_t) = (\ul \delta_{0}, \ul \delta_1)$.
Since $\wh Y = \ol Y_{\ul \gamma}$, we obtain a rel.\ $Y$ cellular homotopy equivalence
$g' \colon (\wh Y, Y) \to (\ol Y_{\ul \delta}, Y)$.
Now we set $\ol Y :=  \ol Y_{\ul \delta_0}$, and we consider $\ol Y_{\ul \delta}$ as obtained from 
$\ol Y_{\ul \delta_0}$ by attaching the $(k{+}1)$-cells indexed by $\ul \delta_1$.
Hence there is a rel.\ $\ol Y$ cellular homotopy equivalence
$g'' \colon (\ol Y_{\ul \delta}, \ol Y) \to (\ol Y \vee (\vee_{i=1}^l S^{k+1}), \ol Y)$.
We set $g \coloneqq  g'' \circ g'$, let $g^{-1}$ be a homotopy inverse for $g$ 
and define $f' : = f \circ (g^{-1}|_{\vee_{i=1}^l S^{k+1}})$.

(2) Consider the cellular chain complexes of $Y$, $\ol Y$ and $\wh Y$:  
$C_*(Y)$ and $C_*(\wh Y)$ are $\Z V$-module chain complexes while $C_*(\ol Y)$ is a chain complex of free abelian groups.  We have
\[ 
 C_j(\ol Y) \cong 
\begin{cases}
C_j(Y) & j \neq k{+}1, \\
C_j(Y) \oplus M & j = k{+}1,
\end{cases}  
\quad \text{and} \quad
C_j(\wh Y) \cong 
\begin{cases}
C_j(Y) & j \neq k{+}1, \\
C_j(Y) \oplus F & j = k{+}1.
\end{cases}  
\]
Since $Y$ has no odd-dimensional cells, the differentials of $C_*(Y)$ vanish
and consequently $C_*(Y) = H_*(Y)$. Hence we have the following commutative diagram,
\begin{equation}\label{eq:diagM}
    \begin{tikzcd}
        M\arrow[r, hook]&\pi_{k+1}(Z,Y)\arrow[r,"\del"]&\pi_k(Y)\arrow[rr,"\rho"]\arrow[d] & &H_k(Y)\arrow[d,"\cong"]\\
        & & C_{k+1}(\ol Y)\arrow[r,"d"]&C_k(\ol Y)\arrow[r,"="]& C_k(Y),
    \end{tikzcd}
\end{equation}
where the vertical arrow on the left is defined by taking $\del(\delta_i)$ to the cell with attaching map  $\del(\delta_i)$. Hence the only non-trivial differential in $C_*(\ol Y)$
is
\[d^{\ol Y}_{k+1} \colon C_{k+1}(\ol Y)\cong C_{k+1}(Y) \oplus M \to C_k(Y),\]
which vanishes on $C_{k+1}(Y)$
and $d^{\ol Y}_{k+1}|_M$ is injective by \ref{ite:M3}.
It follows that
\[H_{k+1}(\ol Y) \cong C_{k+1}(Y) = H_{k+1}(Y).\]


(3) For the weak $V$-$CW$-complex $\wh Y$, the only non-trivial differential in $C_*(\wh Y)$
is $d^{\wh Y}_{k+1} \colon C_{k+1}(Y) \oplus F\to C_k(Y)$, which vanishes on $C_{k+1}(Y)$
and is identified with $\rho\cdot\del\cdot p$ on $F$, since $C_k(Y) \cong H_k(Y)$.
It follows that there is a splitting of $\Z V$-modules 
$H_{k+1}(\wh Y) \cong H_{k+1}(Y) \oplus K$ and using Part~(2) to identify
$H_{k+1}(Y) = H_{k+1}(\ol Y)$ and endow $H_{k+1}(\ol Y)$ with a $\Z V$-module
structure, we have the splitting of $\Z V$-modules
$H_{k+1}(\wh Y) \cong H_{k+1}(\ol Y) \oplus K$.
Since $g$
is a cellular homotopy equivalence, the underlying splitting of abelian
groups is reflected in the isomorphism 
$g_* \colon H_{k+1}(\wh Y) \to 
H_{k+1}(\ol Y \vee \left( \vee_{j = 1}^l S^{k+1} \right)) = H_{k+1}(\ol Y) \oplus K$. 
Since
$H_{k+1}(Y)$ and $K$
are $\Z V$-submodules of $ H_{k+1}(\wh Y)$, 
and the proof is complete.

(4) The space
$\ol Y \vee \left( \vee_{j = 1}^l S^{k+1} \right)\simeq\ol Y_{\ul\delta}$ is constructed by attaching one $(k{+}1)$-cell for each $\delta_i'$. We abuse notation and let $\delta_i'$ also denote the class in 
\[\pi_{k+1}(\ol Y \vee \left( \vee_{j = 1}^l S^{k+1} \right))\cong\pi_{k+1}(\wh Y)\]
represented by the attaching map of this cell. Then the summand $K$ in $H_{k+1}(Y)\oplus K$ is generated by $\rho(\delta_i')$ for $i\geq n{+}1$.
\end{proof}
\begin{proof}[Proof of Proposition~\ref{pro:$CW$-decomp}]
    Recall the $W$-equivariant map 
$f \colon P^{\times n} \to \BM T^{n-1}$ of Lemma~\ref{lem:Bq-input} (2),
where $W = S_n$ acts on $P^{\times n}$ by permuting coordinates
and $P$ is a $CW$-complex with
one cell in every even dimension
and no odd-dimensional cells.
We equip $P^{\times n}$ with the product $CW$-structure,
so that $P^{\times n}$ is a weak $W$-$CW$-complex with a single $0$-cell, which is $W$-fixed.
As we said above, we will use the $W$-equivariant map $f$ as a guide for adding $W$-free cells to 
$P^{\times n}$ to create a weak $W$-$CW$-complex model for $BT_{PU_n}$, 
and we denote this model by $X$.

We claim that $f_*\colon H_*(P^{\times n})\to H_*(\BM T^{n-1})$ is surjective. Indeed, since the map $f\colon P^{\times n}\to \BM T^{n-1}$ is a homomorphism of H-spaces, the induced map \[f_*\colon H_*(P^{\times n})\to H_*(B_MT^{n-1})\]
is a homomorphism of divided power algebras. Furthermore, we note that $f_*$ is surjective in dimension $2$, and that $H_*(B_MT^{n-1})$ is generated as a divided power algebra by elements in $H^2(B_MT^{n-1})$. It follows that 
$f_*\colon H_*(P^{\times n})\to H_*(B_MT^{n-1})$
is surjective. 
Hence we will aim to kill the
homology groups $\Ker(f_*)$ inductively.

Starting from the weak $W$-$CW$-complex $P^{\times n}$,
we denote the $k$-skeleton of the relative $W$-$CW$-complex model for 
$(BT_{PU_n}, P^{\times n})$ by $(X^{(k)}, P^{\times n})$.  Hence
\[ X^{(k)} \coloneqq  P^{\times n} \bigcup( \cup_i e_i), \]
where $e^i$ is a cell of dimension $k$ or less,
and we denote the restriction of the map $B_MT^n\to B_MT^{n-1}$ to $X^{(k)}$ by $f^{(k)}$. We will build $X^{(k+1)}$ by inductively adding $W$-free $k$-cells to $X^{(k)}$ so that the map
$f \colon P^{\times n} \to \BM T^{n-1}$
extends to a $(k{+}1)$-connected map
\[ f^{(k+1)} \colon X^{(k+1)} \to \BM T^{n-1}.\]
We will require the following inductive hypotheses on $X^{(k)}$ and $f^{(k)}$:

\begin{enumerate}[label=\textbf{X\arabic*}]
\item\label{ite:x1} The map $f^{(k)}$ is $W$-equivariant.

\item\label{ite:x2} There is a (non-equivariant) homotopy equivalence $X^{(k)} \simeq B^{(k)} \vee (\vee_{i=1}^{l_k} S^k)$,
where $l_k$ is a non-negative integer, 
$B^{(k)} = P^{\times n} \cup (\cup_j e_j)$ with the $e_j$ cells of dimension $\leq k$ and
\[  
\begin{cases}
(f^{(k)}|_{B^{(k)}})_* \colon H_i(B^{(k)}) \to H_i(\BM T^{n-1}) & \text{is an isomorphism for
$i < k$}, \\
H_i(P^{\times n}) \to H_i(B^{(k)}) & 
\text{is an isomorphism for $i \geq k$.}
\end{cases} \]
In particular, $f^{(k)}$ is $k$-connected. 
\item\label{ite:x3} The splitting $X^{(k)} \simeq B^{(k)} \vee (\vee_{i=1}^{l_k} S^k)$ 
of \ref{ite:x2} induces a splitting of $\Z W$-modules 
\[ H_k(X^{(k)}) \cong H_k(P^{\times n}) \oplus \wt A_k,\]
where $\wt A_k \cong \Z^{l_k}$ is a finitely generated $\Z W$-module 
contained in the image of the Hurewicz homomorphism
$\rho \colon \pi_k(X^{(k)}) \to H_k(X^{(k)})$.
\end{enumerate}
We note that these inductive hypotheses imply that $H_*(X^{(k)})$ is finitely generated torsion--free and also that 
\begin{equation}\label{eq:f_A_k}
    f^{(k)}_*({\wt A_k}) = 0,~\text{for all $k > 2$},
\end{equation}
since $\pi_i(BT^{n-1}) = 0$ for $i > 2$.

To start, we note that $f$ is a $2$-equivalence, 
and so we take $X^{(1)} = X^{(2)} = P^{\times n}$ and $f^{(1)} = f^{(2)} = f$,
which trivially satisfy \ref{ite:x1} above. Taking $B^{(1)} = B^{(2)} = P^{\times n}$ and $\wt A_1 = \wt A_2 =0$, the hypotheses \ref{ite:x2} and \ref{ite:x3} are satisfied.

For $k \geq 2$, suppose that we have constructed $f^{(k)} $ which satisfies the inductive hypotheses \ref{ite:x1}, \ref{ite:x2} and \ref{ite:x3}.
We regard $f^{(k)}$ as an inclusion and consider the pair 
$(\BM T^{n-1}, X^{(k)})$.
%
%
The homotopy and homology long exact sequences of
this pair give the following commutative diagram:

\begin{equation}\label{eq:Hurewicz}
    \begin{tikzcd}
        0 \arrow[d] \arrow[r] &\pi_{k+1}(\BM T^{n-1}, X^{(k)}) \arrow[d, "\rho", "\cong"']\arrow[r,"\del"] &
\pi_k(X^{(k)}) \arrow[d, "\rho"]\arrow[r] & \pi_k(B_MT^{n-1}) \arrow[d] \\
H_{k+1}(\BM T^{n-1}) \arrow[r, "0"]& H_{k+1}(\BM T^{n-1}, X^{(k)})  \arrow[r, "\del", hook]&
H_k(X^{(k)}) \arrow[r,"f_*^{(k)}", twoheadrightarrow] & H_k(\BM T^{n-1}).
    \end{tikzcd}
\end{equation}
Here, the vertical maps are Hurewicz homomorphisms. Since $f^{(k)}$ is a $k$-equivalence, 
$\rho \colon \pi_{k+1}(\BM T^{n-1}, X^{(k)}) 
\to H_{k+1}(\BM T^{n-1}, X^{(k)})$
is an isomorphism by the Relative Hurewicz 
Theorem~\cite[Chapter IV, Theorem 7.2 ]{whitehead1978}.
Since $P^{\times n} \subset X^{(k)}$, 
the induced homomorphism $f_*^{(k)}\colon H_*(X^{(k)}) \to H_*(\BM T^{n-1})$ is onto,
and this implies that $H_{k+1}(\BM T^{n-1}) \to H_{k+1}(\BM T^{n-1}, X^{(k)})$ is the zero map.

Recall the following isomorphism in \ref{ite:x3}:
\begin{equation}\label{eq:x3}
   H_k(X^{(k)}) \cong H_k(P^{\times n}) \oplus \wt A_k
\end{equation}
such that $f^{(k)}_*({\wt A_k}) = 0$, by \eqref{eq:f_A_k}. Therefore, by \eqref{eq:Hurewicz}, we have $\wt A_k\subset\opn{Im}\del\cdot\rho$ and that  $\rho\cdot\del$ is injective.  A diagram chase then yields
\begin{equation}\label{eq:Im_rho_del} 
    \begin{split}
        &\pi_{k+1}(B_MT^{n-1}, X^{(k)})\xrightarrow[\cong]{\rho\del} \opn{Im}\{\rho\del\colon \pi_{k+1}(B_MT^{n-1}, X^{(k)})\to H_k(X^{(k)})\}\\
        =& \opn{Ker}\{f_*\colon H_k(P^{\times n})\to H_k(B_M T^{n-1})\}\oplus \wt A_k.
    \end{split}
\end{equation}
Therefore, there are free $\Z V$-modules $A_{k+1}'$, $A_{k+1}''$, and a $\Z V$-module homomorphism 
\begin{equation}\label{eq:A_k_def}
    p\colon A_{k+1} = A_{k+1}'\oplus A_{k+1}''\to \pi_{k+1}(B_M T^{n-1, X^{(k)}})
\end{equation}
such that
\begin{equation}\label{eq:Ak}
    \rho\cdot\del\cdot p(A_{k+1}') = \opn{Ker}\{H_k(P^{\times n})\xrightarrow{f_*}H_k(B_M T^{n-1})\},\  \rho\cdot\del\cdot p(A_{k+1}'') = \wt A_k.
\end{equation}

Let $\ul \alpha'' = (\alpha_1'',\cdots,\alpha_{s_1}'')$ be a $\Z V$-basis for $A_{k+1}'$, and let $\ul\alpha' = (\alpha_1',\cdots,\alpha_{t_1}')$ be the sequence formed by members of the form $w\alpha_i''$ for $w\in W$ and $1\leq i\leq s_1$.

Let $\ul\beta'' = (\beta_1'',\cdots,\beta_{s_2}'')$ be a $\Z V$-basis for $A_{k+1}''$, and similarly define $\ul\beta' = (\beta_1',\cdots,\beta_{t_2}')$.

Let $\alpha_i = p(\alpha_i')$, $\beta_i = p(\beta_i')$. We proceed making use of Lemma~\ref{l:W-cells}. Let
\begin{equation}\label{eq:gamma}
    \ul\gamma = (\gamma_1,\cdots,\gamma_t) = (\alpha_1,\cdots,\alpha_{t_1},\beta_1,\cdots,\beta_{t_2}),    
\end{equation}
and let 
\begin{equation}\label{eq:MY}
    Y = \XX{k},\ Z = \BM T^{n-1},\ M = \pi_k(Z,Y).
\end{equation}
We define 
\[\XX{k+1} = \wh Y,\ \BB{k+1} = \ol Y.\]
\noindent
The Diagram \eqref{eq:Hurewicz} shows that conditions \ref{ite:M1} and \ref{ite:M3} are satisfied.

By construction, $f^{(k+1)}$ is $W$--equivariant, so \ref{ite:x1} holds for $k{+}1$. 
For \ref{ite:x2}, we have 
\begin{equation*}
    f^{(k+1)}_*\colon H_i(B^{(k+1)})\cong H_i(X^{(k)})\cong H_i(B^{(k)})\xrightarrow[\cong]{f^{(k)}_*}H_i(B_MT^{n-1}),\ i \neq k, k+1.
\end{equation*}
We have the first isomorphism since $B^{(k+1)}$ is obtained by attaching $(k{+}1)$-cells to $X^{(k)}$. The second isomorphism follows since $X^{(k)}$ is the wedge sum of $B^{(k)}$ and copies of $S^k$. The third isomorphism follows from \ref{ite:x2} as part of the induction hypothesis.

It remains to consider the cases $i = k, k{+}1$. For $i=k$, by Lemma~\ref{l:W-cells}, (1), we have a homotopy equivalence $\XX{k+1}\simeq\BB{k+1}\vee (\vee_{i=1}^{l_{k+1}} S^{k+1})$. Furthermore, by \eqref{eq:x3}, and \eqref{eq:Ak}, we have 
\begin{equation}\label{eq:x2_k+1_a}
    \begin{split}
        H_k(\BB{k+1})&\cong H_k(\XX{k})/\del\cdot\rho(M)\\
        &\cong H_k(P^{\times n})/\opn{Ker}(f_*)\oplus \wt A_k/\wt A_k\\
        &\cong H_k(\BM T^{n-1}).
    \end{split}
\end{equation}
For $i= k{+}1$, we have
\begin{equation}\label{eq:x2_k+1_b}
    H_{k+1}(P^{\times n})\xrightarrow{\cong}H_{k+1}(\BB{k})\xrightarrow{\cong}H_{k+1}(\XX{k})\xrightarrow{\cong}H_{k+1}(\BB{k+1}),
\end{equation}
where the first and second isomorphisms follow from \ref{ite:x2} for $k$, and the third one from Lemma~\ref{l:W-cells}, (2). By \eqref{eq:x2_k+1_a} and \eqref{eq:x2_k+1_b}, \ref{ite:x2} holds for $k{+}1$. 

We consider \ref{ite:x3}. By Lemma~\ref{l:W-cells}, (3) and (4), we have an isomorphism of $\Z W$--modules
\begin{equation}\label{eq:x3_k+1_a}
    H_{k+1}(\XX{k+1})\cong H_{k+1}(\XX{k})\oplus\wt A_{k+1}
\end{equation}
such that $\wt A_{k+1}$ is a finitely generated $\Z W$-module in the image of the Hurewicz homomorphism $\rho$. In addition, by \eqref{eq:x2_k+1_b}, we have the isomorphism $H_{k+1}(\XX{k})\cong H_{k+1}(\BB{k+1})$ induced by the inclusion. Therefore, \ref{ite:x3} holds for $k+1$.

We define $(X, P^{\times n})$ to be the infinite relative $CW$-complex obtained
by attaching free $W$-cells in all dimensions:
\[ X \coloneqq  \bigcup_{k \to \infty} X^{(k)}.\]
Finally, we show that the weak $W$-$CW$-complex $X$ is a $\varphi$-adapted model for $BT_{PU_n}$.
By the construction of $X$, the $W$-equivariant map 
$f \colon P^{\times n} \to \BM T^{n-1}$ factors as
\[ P^{\times n} \xrightarrow{~\iota~} X \xrightarrow{~f^\infty~} \BM T^{n-1},\]
where the first arrow is the canonical inclusion 
$\iota \colon P^{\times n} \hookrightarrow X$ and $f^\infty$ is the union of the maps $W$-equivariant maps $f^{(k)} \colon X^{(k)} \to \BM T^{n-1}$ and is thus $W$-equivariant.
Since $f^\infty$ is the increasing union of $k$-connected maps between simply-connected spaces,
$f^\infty$ is a weak homotopy equivalence.
%
By Lemma~\ref{lem:Bq-input} (1), the $W$-action on $\BM T^{n-1}$ is $\psi$-adapted, and so applying Lemma~\ref{lem:easy_obs} to $f^\infty \colon X \to \BM T^{n-1}$, 
we conclude that the $W$-action on $X$ is also $\psi$-adapted.

Next, we consider the projection $Z^k(X)^W\to H^k(X)^W$. For the chain complex $C_*(X)$, let $d_k\colon C_k(X)\to C_{k-1}(X)$ denote the differential. We have
\begin{equation*}
    C_k(X) = C_k(\XX{k})\cong C_k(P^{\times n})\oplus A_k,
\end{equation*}
where $A_k$ is defined as in \eqref{eq:A_k_def}. For any $w\in W$, $d_k(w\alpha_i')$ (or $d_k(w\beta_j')$) is a cocycle representing $\del\cdot\rho(\alpha_i)$ (or $\del\cdot\rho(\beta_j)$). 
Therefore we have $A_k = A_k'\oplus A_k''$ and 
\begin{equation}\label{eq:directsum}
    d_k\colon C_k(P^{\times n})\oplus A_k'\oplus A_k''\to C_{k-1}(P^{\times n})\oplus A_{k-1}.
\end{equation}
With respect to the direct sum decompositions in  \eqref{eq:directsum},
$d_k$ is represented by a matrix
\begin{equation}\label{eq:im_u'_im_v'}
    d_k = 
    \begin{pmatrix}
        0 & u_k' & 0\\
        0 &  0  &v_k'
    \end{pmatrix},
\end{equation}
where 
$u_k'\colon A_k' \to C_{k-1}(P^{\times n})$
and
$v_k'\colon A_k''  \to A_{k-1}$
and the $0$'s are $0$ matrices over $\Z$. By \eqref{eq:Ak} we have $d_k(\alpha_i') = u_k'(\alpha_i')$ and $d_k(\beta_i') = v_k'(\beta_i')$. 

Consider the compositions
\[u_k\colon C_k(X)\to A_k'\xrightarrow{u_k'}C_{k-1}(P^{\times n})\]
and 
\[v_k\colon C_k(X)\to A_k''\xrightarrow{v_k'}A_{k-1}\]
where the unlabeled arrows denote projections onto direct summands. By \eqref{eq:im_u'_im_v'}, we have
\begin{equation}\label{eq:im_u_im_v}
    \begin{cases}
        d_k = u_k + v_k,\\
        \opn{Ker}d_k = \opn{Ker}u_k \oplus \opn{Ker}v_k,\\
        \opn{Im} d_k = \opn{Im} u_k \oplus \opn{Im} v_k.
    \end{cases}     
\end{equation}
\noindent
Since the homomorphism
\[\iota_{\#}\colon C_k(P^{\times n})\hookrightarrow C_k(P^{\times n})\oplus A_k = C_k(X)\]
induces a surjective homomorphism between homology groups, by \eqref{eq:im_u_im_v}, we have an exact sequence 
\begin{equation}\label{eq:ses_v}
    C_{k+1}(X)\xrightarrow{v_{k+1}} A_k\xrightarrow{v_k}A_{k-1}
\end{equation}
and
\begin{equation}\label{eq:d_kC}
   d_k(C_k(P^{\times n})) = 0.
\end{equation}

The situation is illustrated as follows:
\begin{equation}\label{eq:C(X)}
    \begin{tikzcd}
        C_{2j+2}(BT^n) & \oplus & \{0\} & \oplus & A_{2j+2}'' \arrow[dll] \arrow[d] \\
\{0\} & \oplus & A_{2j+1}' \arrow[dll,"u_{2j+1}"]  & \oplus & A_{2j+1}'' \arrow[d, "v_{2j+1}"] \\
C_{2j}(BT^n) & \oplus & \{0\} & \oplus & A''_{2j}
\arrow[dll,"v_{2j}"] \arrow[d, "v_{2j}"]  \\
\{0\} & \oplus & A_{2j-1}' & \oplus & A''_{2j-1}. 
    \end{tikzcd}
\end{equation}

For an abelian group $A$, let $\check{A} = \opn{Hom}(A,\Z)$. For a homomorphism of abelian groups $\phi\colon A\to B$, let $\check{\phi}\colon\check{B}\to\check{A}$ denote its dual. In particular, we have 
\begin{equation}\label{eq:cochain_direct_sum}
    C^k(X) \cong C^k(P^{\times n})\oplus \check{A}_k.
\end{equation}
Taking the transpose of the matrix \eqref{eq:im_u'_im_v'}, we have the dual of \eqref{eq:im_u_im_v} 
\begin{equation}\label{eq:im_u_im_v_dual}
    \begin{cases}
        \check{d}_k = \check{u}_k + \check{v}_k,\\
        \opn{Ker}\check{d}_k = \opn{Ker}\check{u}_k \oplus \opn{Ker}\check{v}_k,\\
        \opn{Im} \check{d}_k = \opn{Im} \check{u}_k \oplus\opn{Im} \check{v}_k.
    \end{cases}     
\end{equation}
By \eqref{eq:d_kC} and \eqref{eq:im_u_im_v_dual}, we have
\begin{equation}\label{eq:d_kC_dual}
    \opn{Im}\check{d}_k \subseteq \check{A}_k.
\end{equation}
\noindent
Since $H_k(X)\cong H_k(\BM T^{n-1})$ is concentrated in even degrees, the projection $Z^k(X)^W\to H^k(X)^W$ is trivially surjective when $k$ is odd. 

The remaining case is $k =2m$. By \eqref{eq:ses_v}, we have a short exact sequence
\begin{equation}\label{eq:SE_A}
   C_{2m+1}(X)\xra{v_{2m+1}}A_{2m}\xra{d_{2m}}B_{2m-1}(X)\to 0.
\end{equation}
\noindent
Then \eqref{eq:SE_A} yields a short exact sequence
\begin{equation}\label{eq:SE_Adual}                 
  C^{2m+1}(X)\xleftarrow{\check{v}_{2m+1}}\check{A}_{2m}\xleftarrow{\check{d}_{2m}}\check{B}_{2m-1}(X)\leftarrow 0.
\end{equation}
Since $H_{2m-1}(X) = 0$, we have $B_{2m-1}(X) = Z_{2m-1}(X)$. Therefore, we have
\begin{equation}\label{eq:B=Z}
    \check{B}_{2m-1}(X) = \check{Z}_{2m-1}(X).
\end{equation}
Since $C_{2m-1}(X) = A_{2m-1}$, we have another short exact sequence
\begin{equation}\label{eq:SES_ZAB}
    0\to Z_{2m-1}(X)\hookrightarrow A_{2m-1}\xra{d_{2m-1}} B_{2m-2}(X)\to 0,
\end{equation}
where $B_{2m-2}$ is torsion--free. Therefore, the short exact sequence \eqref{eq:SES_ZAB} splits and we have yet another short exact sequence
\begin{equation}\label{eq:SES_ZA}
    0\leftarrow \check{Z}_{2m-1}(X)\leftarrow \check{A}_{2m-1}.
\end{equation}
By \eqref{eq:B=Z} and \eqref{eq:SES_ZA}, we have
\begin{equation}\label{eq:Images}
    \opn{Im}\{\check{B}_{2m-1}(X)\xra{\check{d}_{2m}}\check{A}_{2m}\} = 
    \opn{Im}\{\check{C}_{2m-1}(X)=\check{A}_{2m-1}\xra{\check{v}_{2m}}\check{A}_{2m}\}.
\end{equation}
By \eqref{eq:SE_Adual} and \eqref{eq:Images}, we have one more short exact sequence
\begin{equation}\label{eq:SESdual}
    C^{2m+1}(X)\xleftarrow{\check{v}_{2m+1}}\check{A}_{2m}\xleftarrow{\check{v}_{2m}}\check{A}_{2m-1}.
\end{equation}
By \eqref{eq:cochain_direct_sum}, we have
\begin{equation}\label{eq:cochain_direct_sum_2m}
   C^{2m}(X)\cong C^{2m}(P^{\times n})\oplus\check{A}_{2m},\ C^{2m-1}(X) = \check{A}_{2m-1}.
\end{equation}
Assembling \eqref{eq:SESdual} and \eqref{eq:cochain_direct_sum_2m}, and applying \eqref{eq:im_u_im_v_dual}, we deduce that the inclusion $C^{2m}(P^{\times n})\hookrightarrow C^{2m}(X)$ in the sense of \eqref{eq:cochain_direct_sum} induces an isomorphism
\begin{equation}\label{eq:H2m}
    H^{2m}(X)\cong \opn{Ker}\{\check{u}_{2m+1}\colon C^{2m}(P^{\times n})\to C^{2m+1}(X)\}.
\end{equation}

By \eqref{eq:im_u_im_v_dual} and \eqref{eq:H2m}, we have
\begin{equation}\label{eq:Z2m}
    Z^{2m}(X) \cong \opn{Ker}\check{u}_{2m+1} \oplus \opn{Ker}\check{v}_{2m+1}\cong H^{2m}(X) \oplus \opn{Ker}\check{v}_{2m+1}.
\end{equation}
It follows that the projection 
\begin{equation*}
    Z^{2m}(X)^W \cong H^{2m}(X)^W \oplus (\opn{Ker}\check{v}_{2m+1})^W\to H^{2m}(X)^W
\end{equation*}
is the projection onto a direct summand, which is onto.
\end{proof}

\numberwithin{equation}{subsection}
\numberwithin{theorem}{subsection}

\section{On the Weyl group invariants $H^*(BT_{PU_n})^W$} \label{s:primitive}
In the previous section we proved that $\rho_{PU_n} \colon H^*(BPU_n) \to H^*(BT_{PU_n})^W$ is onto.
In this section we present several descriptions of $H^*(BT_{PU_n})^W$ and $\rho_{PU_n}$,
proving Corollary \ref{cor:Bq} and Theorem \ref{thm:6groups}.
First we summarise what was known rationally.

Let $SU_n \subset U_n$ denote the special unitary group.  
There is a normal $(\Z/n)$-covering space
$\Z/n \to SU_n \xra{q} PU_n $
and a corresponding fibration sequence
\[ B(\Z/n) \to BSU_n \xra{Bq} BPU_n. \]
Since the fibre $B(\Z/n) = K(\Z/n, 1)$ has trivial rational cohomology, the Serre spectral sequence of the above fibration shows that
\[ (Bq)^* \colon H^*(BPU_n; \Q) \to H^*(BSU_n; \Q) = \Q[c_2, \dots, c_n] \]
is an isomorphism. 
Since $H^*(BPU_n; \Q) \to H^*(BT_{PU_n}; \Q)^W$ is an isomorphism,
by, for example, \cite[Chapter III, Lemma 1.1, Reduction 2]{hsiang2012cohomology}, we have
\begin{proposition} \label{prop:H^*(BPU_n;Q)}
$H^*(BT_{PU_n}; \Q)^W \cong H^*(BPU_n; \Q) \cong \Q[c_2, \dots, c_n]$.  \qed
\end{proposition}

\subsection{The isomorphism $FH^*(BPU_n) \to H^*(BT_{PU_n})^W$}
For any space $X$, we let $TH^*(X) \subseteq H^*(X)$ denote the ideal of torsion classes and define the free quotient of $H^*(X)$ by $FH^*(X) \coloneqq  H^*(X)/TH^*(X)$.
Since $H^*(BT_{PU_n})^W$ is torsion--free,
Theorem \ref{thm:mainPU} implies that $\rho_{PU_n}$ induces
an isomorphism
\begin{equation}\label{eq:mod-torsion}
\overline \rho_{PU_n} \colon FH^*(BPU_n) \to H^*(BT_{PU_n})^W.
\end{equation}
It is interesting to ask whether there is a section of $\overline \rho_{PU_n}$ which is a ring homomorphism. When $n = 3$, Vistoli \cite{vistoli2007cohomology}, constructs such a section, completing the work of Vezzosi \cite{vezzosi2000chow}.

\subsection{The isomorphism $ H^*(BT_{PU_n})^W \to \opn{Ker}(\nabla_{\! n})$}
Recall the fibration sequence
\[ BS^1 \to BU_n \xra{Bq} BPU_n,\]
which is obtained by applying the classifying space functor to the central extension
\begin{equation} \label{eq:UtoPU}
S^1 \to U_n \xra{q} PU_n.
\end{equation}
We will investigate the image of the homomorphism $(Bq)^* \colon H^*(BPU_n) \to H^*(BU_n)$. 

Without risk of ambiguity, we denote the restriction of $q$ on the maximal torus also by $q$:
\[q \colon T_{U_n}\to T_{PU_n},\]
and this map may be identified with the quotient map
$T^n \to T^{n-1} = T^n/S^1$, where $S^1 \subset T^n$ is the diagonal circle.  The central extension \eqref{eq:UtoPU} induces an isomorphism of Weyl groups, $W_{U_n} \to W_{PU_n}$, and there is a canonical identification $W_{U_n} = S_n$.
We consider the following commutative diagram:
%
\begin{equation} \label{eq:Bq}
    \begin{tikzcd}
        H^*(BU_n)\arrow[r, "\rho_{U_n}", "\cong"'] & H^*(BT_{U_n})^W\arrow[r, "="] & \Z[v_1, \dots, v_n]^{S_n}\\
        H^*(BPU_n)\arrow[u, "Bq^*"]\arrow[r, "\rho_{PU_n}"]& H^*(BT_{PU_n})^W\arrow[u, "Bq^*"]\arrow[r,,"="]&
        \Z[w_1, \dots, w_{n-1}]^{S_n}\arrow[u,"Bq^*"]. 
    \end{tikzcd}
\end{equation}
Here, as later justified by Lemma~\ref{lem:WI-Alg}, the generators $w_i$ and $v_i$ may be chosen to that they satisfy $(Bq)^*(w_i) = v_{i+1}-v_i$. It is well known that $\rho_{U_n}$ is an isomorphism, which we
shall regard as an identification. In order to characterize the image of $(Bq)^*$, we define
the derivation
\[ \nabla_{\! n} := \sum_{i=1}^n \frac{\partial}{\partial v_i} \colon \Z[v_1, \dots, v_n] \to \Z[v_1, \dots, v_n].\]
Since $\nabla_{\! n}$ commutes with the action of $S_n$ on $\Z[v_1, \dots, v_n]$, it restricts to the symmetric polynomials and we use the same symbol to denote the restriction
\[ \nabla_{\! n} \colon \Z[v_1, \dots, v_n]^{S_n} \to \Z[v_1, \dots, v_n]^{S_n}.\]
%
Since
$\nabla_n((Bq)^*(w_i)) = \nabla_n(v_{i+1}-v_i) = 0$ for $i = 1, \dots, n{-}1$
and $\nabla_n$ is a derivation,
$\Ker(\nabla_n) \subseteq (Bq)^*(\Z[w_1, \dots, w_{n-1}])^{S_n}$.  
The proof of the opposite inclusion is essentially given in the proof of \cite[Proposition 3.3 (2)]{gu2019cohomology}, where variables $u_i = v_i - v_n$ appear, and we adapt that proof to the variables $w_i$.

\begin{lemma} \label{lem:WI-Alg}
We have
$(Bq)^*(\Z[w_1, \dots, w_{n-1}]^{S_n}) = \opn{Ker}(\nabla_{\! n})$. In particular, $\opn{Ker}(\nabla_{\! n})$ is an $S_n$-invariant subring of $\Z[v_1,\cdots,v_n]$.
%
\end{lemma}

\begin{proof}
Let $w_i' = v_{i+1}-v_i$. It suffices to show for any polynomial $\theta(v_1,\dots,v_n)$ (not necessarily $S_n$-invariant), that if $\nabla(\theta(v_1,\dots,v_n)) = 0$, then \[\theta(v_1,\dots,v_n) = \wt{\theta}(w_1',\dots,w'_{n-1},v_n)\]
for some polynomial $\wt{\theta}$.
%
Changing variables, let
$\theta(v_1, \dots, v_n) =
\wt{\theta}(w_1',\dots,w_{n-1}', v_n)$.
Applying the Chain Rule, we have
\begin{multline*}
\nabla\bigl( \theta(v_1, \cdots, v_n) \bigr) 
    = \frac{\partial \tilde \theta}{\partial v_1} +
    \frac{\partial \tilde \theta}{\partial v_2}
    + \dots + 
    \frac{\partial \tilde \theta}{\partial v_{n-1}}+
    \frac{\partial \tilde \theta}{\partial v_n} =
    \\
    \frac{-\del \tilde \theta}{\del w_1'} + 
    \left(\frac{\del \tilde \theta}{\del w_1'} - \frac{\del \tilde \theta}{\del w_2'}\right) + 
    \dots + 
    \left(\frac{\del \tilde \theta}{\del w_{n-2}'} - \frac{\del \tilde \theta}{\del w_{n-1}'}\right) +
    \left(\frac{\del \tilde \theta}{\del w_{n-1}'} + \frac{\del \tilde \theta}{\del v_n}\right)
     = 
     \frac{\partial \tilde \theta}{\del v_n}.
\end{multline*}
Now, $\tilde \theta(w_1', \dots, w_{n-1}', v_n) \in \Z[w_1, \dots, w_n]$ if and only if $\displaystyle \frac{\partial \tilde \theta}{\del v_n} = 0$, 
and therefore $\nabla(\theta(v_1, \dots, v_n)) = 0$ if and only if $\theta(v_1, \dots, v_n) \in \Z[w_1', \dots, w_{n-1}']$, as required.
\end{proof}

\noindent
Since $(Bq)^*$ is injective, Lemma \ref{lem:WI-Alg} gives

\begin{lemma} \label{lem:WI-Alg2}
$(Bq)^* \colon H^*(BT_{PU_n})^W \to \opn{Ker}(\nabla_{\! n})$ is an isomorphism. \qed
\end{lemma}

Since $H^*(BT_{U_n})^W \cong H^*(BU_n)$,
we also regard $\nabla_{\! n}$ as a derivation,
\begin{equation}\label{eq:nabla}
\nabla_{\! n}\colon H^*(BU_n)\to H^{*-2}(BU_n),
\end{equation}
which was investigated by the second author in \cite{gu2019cohomology}.
If $c_k \in \Z[v_1, \dots, v_n]^{S_n}$ denotes the $k$th elementary symmetric polynomial in $v_1,\cdots,v_n$,
equivalently the $k$th Chern class, then an elementary calculation gives
\begin{equation}\label{eq:nabla_c}
    \nabla_{\! n}(c_k) = (n{-}k{+}1)c_{k-1},
\end{equation}
and along with the Leibniz rule, this identity determines $\nabla_{\! n}$.
Now we observe that $\opn{Ker}(\nabla_{\!n}) \subset H^*(BU_n)$ is a summand, since
it is the kernel of a homomorphism
to the torsion--free group $H^{*-2}(BU_n)$.
Lemma~\ref{lem:WI-Alg2}, Theorem~\ref{thm:mainPU} and diagram~\eqref{eq:Bq}
now give the following algebraic characterisation of $(Bq)^*(H^*(BPU_n)) \subset H^*(BU_n)$.

\begin{proposition} \label{prop:Bq-Alg}
We have the identification of subrings $(Bq)^*(H^*(BPU_n)) = \opn{Ker}(\nabla_{\! n})$. Moreover, this subring is a summand of $H^*(BU_n)$. \qed
\end{proposition}


\noindent
Proposition~\ref{prop:Bq-Alg} and Lemma~\ref{lem:WI-Alg2} together prove Corollary \ref{cor:Bq}.

\begin{remark}
From the discussion leading up to Proposition \ref{prop:H^*(BPU_n;Q)}, it was known that $(Bq)^*(H^*(BPU_n)) \subseteq \opn{Ker}(\nabla_{\! n})$
was a subgroup of finite index in each degree.
With Theorem \ref{thm:mainPU} we now know that equality holds.
\end{remark}

\begin{example}
In degree $2$, $\opn{Ker}(\nabla_{\! n}) \cap H^2(BU_n)$ is trivial and in degree $4$
we have the following generator for 
$\bigl( \opn{Ker}(\nabla_{\! n}) \cap H^4(BU_n) \bigr) = (Bq)^*(H^4(BPU_n))$:
%
%
\[
\begin{cases}
2n c_2 - (n{-}1)c_1^2  & \text{$n$ even}, \\
nc_2 - \lfloor n/2 \rfloor c_1^2 & \text{$n$ odd.}
\end{cases} \]
A routine computation shows 
\begin{equation*}
    \begin{cases}
        Bq^* (2n c_2 - (n{-}1)c_1^2) = \sum_{i<j}\sum_{k=i}^n\sum_{l=j}^n w_k w_l,   \text{ $n$ even}, \\
        Bq^*(nc_2 - \lfloor n/2 \rfloor c_1^2) = \frac{1}{2}\sum_{i<j}\sum_{k=i}^n\sum_{l=j}^n w_k w_l, \text{ $n$ odd.}
    \end{cases}
\end{equation*}
\end{example}

\begin{remark} \label{rem:Ker(nabla)}
We consider the ring structure of $\opn{Ker}(\nabla_{\! n})$.
By Proposition~\ref{prop:H^*(BPU_n;Q)},
setting $c_1 = 0$ induces an isomorphism
$\opn{Ker}(\nabla_{\! n}) \otimes \Q \cong \Q[c_2, \dots c_n]$ but
integrally the situation can be more complex.
In the special case $n = 2$, $\opn{Ker}(\nabla_2) = \Z[4c_2 - c_1^2]$ is
a polynomial ring on one generator of degree $4$.
However, for $n = 3$,
Vistoli has shown that $\opn{Ker}(\nabla_3)$ is not a polynomial ring \cite[Theorem 3.7]{vistoli2007cohomology}.
Indeed, there is a class $\delta\in FH^{12}(BPU_3)$ which is not a polynomial in lower dimensional classes, but $27\delta$ is. In general, it is not clear to the authors whether $\opn{Ker}(\nabla_{\! n})$ is even finitely generated as a ring.

Since we have a ring isomorphism $\opn{Ker}(\nabla_{\! n})\cong H^*(BT_{PU_n})^W$, 
we 
are also considering whether $H^*(BT_{PU_n})^W$ is a polynomial ring over $\Z$. For a general compact Lie group $G$, a subgroup $W_0$ of $W$, and a commutative unital ring $R$, it is interesting to consider whether $H^*(BT;R)^{W_0}$ is a polynomial ring over $R$. Ishiguro, Koba, Miyauchi and Takigawa \cite{ishiguro2020modular} studied many cases. For instance, they showed that for $G = SU(n)$ and $W_0 = A_n$, the alternating group, $H^*(BT;R)^{W_0}$ is
\begin{enumerate}
    \item a polynomial ring for $n = 3$ and $R = \F_3$, and
    \item not a polynomial ring for $n = 3$ and $R = \F_p$, $p \neq 3$, and
    \item not a polynomial ring for $n = 4$ and $R = \F_2$.
\end{enumerate}
\end{remark}

\subsection{The equality $PH^*(BU_n) = \opn{Ker}(\nabla_{\! n})$}
In the previous subsection we characterised
\[ (Bq)^*(H^*(BPU_n)) \subset H^*(BU_n) \]
algebraically.
In this subsection we characterise this subgroup topologically.

Recall the map $\mu\colon BS^1\times BU_n\to BU_n$ classifying the external tensor product of the universal complex line bundle over $S^1$ and the universal complex $n$-vector bundle over $BU_n$.
Then $\mu$ induces a homomorphism
\[\mu^* \colon H^*(BU_n)\to H^*(BS^1)\otimes H^*(BU_n),\]
which is a comodule structure for $H^*(BU_n)$ over $H^*(BS^1)$. Notice that $\mu^*$ is also a homomorphism of graded rings. A primitive element of this comodule is an element $c \in H^*(BU_n)$ satisfying $\mu^*(c)=1\otimes c$. The primitive elements of
$H^*(BU_n)$ form a subring
\[ PH^*(BU_n) \coloneqq  \{ c \in H^*(BU_n) \, \mid \, \mu^*(c) = 1 \otimes c \}. \]
%


By the universal property of the map $\mu$, the primitive subring of $H^*(BU_n)$ can be described as the subring of those
characteristic classes of rank $n$ complex vector bundles $E$ which remain unchanged
after tensor product with any line bundle $L$:
\[ PH^*(BU_n) \coloneqq  \{ c \in H^*(BU_n) \, \mid \, c(E) = c(L \otimes E)\} \]

\begin{proposition}\label{prop:PH=Ker(nabla)}
$PH^*(BU_n) = \opn{Ker}(\nabla_{\! n})$. \qed
\end{proposition}

\begin{proof}
 We write $H^*(BS^1)=\Z[v]$ with $v$ of degree $2$. The proposition would follow immediately from 
 \begin{equation}\label{eq:mu ck}
    \mu^*(x)=\sum_{i=0}^{k}v^i\otimes \frac{\nabla_n^i}{i!}(x)
\end{equation}
for all $ x\in H^*(BU_n)$. We show this as follows. By \cite[Proposition 3.2]{toda1987cohomology},  $\mu^*$ 
satisfies
\begin{equation}\label{eq:toda}
    \mu^*(c_k)=\sum_{i+j = k}\binom{n-j}{i} v^i\otimes c_k.
\end{equation}
By \eqref{eq:toda} and the formula \eqref{eq:nabla} for $\nabla_n(c_k)$, \eqref{eq:mu ck} holds for $x = c_k$. For any $x\in H^*(BU_n)$, one only needs to check \eqref{eq:mu ck} for monomials in the elements $c_1\dots,c_n$, and it follows by a simple induction on the number of elements involved.  We leave the details to the reader.

%
%
\end{proof}

\noindent
From Proposition~\ref{prop:PH=Ker(nabla)} and Proposition~\ref{prop:Bq-Alg} we obtain

\begin{proposition}\label{prop:Bq*=PH)}
$(Bq)^*(H^*(BPU_n)) = PH^*(BU_n)$. \qed
\end{proposition}

\begin{remark} \label{rem:pcc}
From the isomorphism $PH^*(BU_n) = \opn{Ker}(\nabla_{\! n})$ and Vistoli's
theorem \cite[Theorem 3.7]{vistoli2007cohomology} (see Remark~\ref{rem:Ker(nabla)}), we know that
$PH^*(BU_n)$ is not in general a polynomial algebra.
Nonetheless, it would be interesting to identify a canonical minimal set of generators
$\wt c_{i,j} \in PH^{2i}(BU_n)$ of
$PH^*(BU_n)$ and then find lifts $\ol c_{i,j} \in H^*(BPU_n)$ of the $\wt c_{i,j}$.
If a natural construction of such classes $\ol c_{i,j}$ could be found, it would be tempting to call
them {\em projective Chern classes}.   For example, $H^4(BPU_n) \cong \Z$ for all $n \geq 2$
and so is generated by a projective Chern class $\ol c_{2,0}$, which is unique up to sign.
\end{remark}


\subsection{The equality $\opn{Ker}(d^{0, *}_3) = \opn{Ker}(\nabla_{\! n})$}
One approach to the computation of $H^*(BPU_n)$ is to use the
Serre spectral sequence associated to
the fiber sequence 
\begin{equation*}
 BU_n\xrightarrow{Bq} BPU_n\to K(\Z,3).
\end{equation*}
We denote the cohomological Serre spectral sequence of this fiber sequence
by $E_*^{*,*}$:
\begin{equation}\label{eq:E2-of-SSS}
 \begin{split}
  & E_2^{s,t}=H^s(K(\Z,3); H^t(BU_n))\cong H^s(\Z,3)\otimes H^t(BU_n)\Rightarrow H^{s+t}(BPU_n),\\
  & d_r^{s,t}\colon E_r^{s,t}\to E_r^{s+r,t-r+1}.
 \end{split}
\end{equation}
Let $x\in H^3(K(\Z,3))$ be the fundamental class. In the spectral sequence $E_*^{*,*}$
of \eqref{eq:E2-of-SSS}, by \cite[Corollary 3.4]{gu2019cohomology} and the Leibniz rule, for $\xi\in H^*(BU_n)$ and $\eta\in H^*(K(\Z,3))$, we have
\[d_{3}( \eta\xi)= \eta\nabla_{\! n}(\xi)x.\]
Hence $\opn{Ker}(d_3^{0, *}) =\opn{Ker}(\nabla_{\! n})$ and so by
Proposition~\ref{prop:Bq-Alg} we have

\begin{proposition}\label{pro:kernel}
$\opn{Ker}(d_3^{0,*})=(Bq)^*(H^*(BPU_n))$.
\end{proposition}

\noindent
Since $E_\infty^{0, *} = (Bq)^*(H^*(BPU_n))$, we obtain
\begin{theorem}
The Serre spectral sequence $E_*^{*,*}$ above satisfies $d_r^{0, *}=0$ for $r>3$.
Hence $\opn{Ker}(\nabla_{\! n}) = E_3^{0, *} = E_{\infty}^{0,*} = (Bq)^*(H^*(BPU_n)) = PH^*(BU_n)$. \qed
\end{theorem}

\numberwithin{equation}{section}
\numberwithin{theorem}{section}

\appendix 
\section{The cellular action on a covering space}\label{sec:covering_space}
Throughout this appendix, $p \colon E\to B$ is a covering space, $F_{b_0} \coloneqq p^{-1}(b_0)$ for $b_0 \in B$ the base point and $U$ denotes the group of deck transformations of $p$.
%
%
For a $CW$-complex $X$, recall that $X^{(n)}$ denotes its $n$-skeleton.

The following lemma (\cite[Proposition 1.34]{hatcher2002algebraic}) is elementary. 

\begin{lemma}\label{lem:UniqueLifting}
Let $X$ be a connected space, and $f\colon X\to B$ a map.  If $\wt{f}_0, \tilde{f}_1 \colon X \to E$ are lifts of $f$ over $p$ such that $\tilde{f}_0(x) = \tilde{f}_1(x)$ for some $x\in X$, then $\tilde{f}_0 = \tilde{f}_1$. \qed
\end{lemma}

\begin{lemma}\label{lem:CellLifting}
    If $B$ is a $CW$-complex, then there is a unique $CW$-complex structure on $E$ such that, if $c$ is a cell of $B$ and $ c^{\circ}$ denotes the corresponding open cell of $c$, then
    there are cells of $E$, denoted by $\{\tilde{c}_i\}_i$, such that
    \begin{enumerate}
        \item\label{ite:a1} 
        we have $p^{-1}(c^{\circ}) = \sqcup_i\tilde{c}_i^{\circ}$, and 
        \item\label{ite:a2} the restriction of $p$ on each $\tilde{c}_i^{\circ}$ is a homeomorphism onto $c^{\circ}$.
    \end{enumerate}
    In particular, for $n\geq 0$, we have $E^{(n)} = p^{-1}(B^{(n)})$.
\end{lemma}
\begin{proof}
    The lemma holds trivially if $B$ is of dimension $0$. For the general case, we argue by induction. Suppose that $E^{(n)} = p^{-1}(B^{(n)})$ has a CW-complex structure satisfying (\ref{ite:a1}) and (\ref{ite:a2}). 
    
    Let $c$ be an $(n+1)$-cell of $B$. Since $c^{\circ}$ is contractible, all covering spaces over $c^{\circ}$ are isomorphic to disjoint unions of copies of itself. Therefore, we have $p^{-1}(c^{\circ}) = \sqcup_i e_i$ where each $e_i$ is an open subset of $E$ homeomorphic to $c^{\circ}$. Let $\tilde{c}_i$ be the closure of $e_i$ in $E$.

    Let $D^{n+1}$ be the closed unit ball of dimension $n+1$. Let $\varphi^c:D^{n+1}\to B$ be the character map of $c$. By Lemma~\ref{lem:UniqueLifting}, $\varphi^c$ has a unique lifting $\tilde{\varphi}^c_i\colon D^{n+1}\to E$ such that $\tilde{\varphi}^c_i$ maps the interior of $D^{n+1}$ homeomorphically onto $e_i$. Let 
     $\tilde{c}_i = \tilde{\varphi}^c_i(D^{n+1})$, and we have $e_i = \tilde{\varphi}^c_i((D^{n+1})^{\circ})$ where $(D^{n+1})^{\circ}$ is the unit open ball of dimension $n+1$.

    We proceed to show that $\tilde{c}_i$ is the closure of $e_i$, and so we have $\tilde{c}_i^{\circ} = e_i$.  The inclusion $\tilde{c}_i\subseteq\opn{cl}(e_i)$ 
    follows immediately from the fact that $D^{n+1}$ is the closure of $(D^{n+1})^{\circ}$ in $D^{n+1}$. For the inclusion of the other direction, let $x\in \opn{cl}(e_i)$. Then  $p(x)\in\opn{cl}(c^{\circ}) = \varphi^c(D^{n+1})$. Therefore, we have $u\in D^{n+1}$, such that $p\tilde{\varphi}^c_i(u) = p(x)$. Let $O$ be a connected neighborhood of $p(x)$ such that $p$ restricts to a trivial covering space over $O$. Let $\tilde{O}_1$ and $\tilde{O}_2$ be connected components of $p^{-1}(O)$ containing $\tilde{\varphi}^c_i(u)$ and $x$, respectively. Since $O\cap c^{\circ}\neq\emptyset$, $\tilde{O}_1$ and $\tilde{O}_2$ both have non-empty intersections with $e_i$. Therefore, we have 
    \[\emptyset\neq \tilde{O}_1\cap e_i = p^{-1}(O\cap c^{\circ}) = \tilde{O}_2\cap e_i,\]
    implying $\tilde{O}_1 = \tilde{O}_2$, and therefore    
    $\tilde{\varphi}^c_i(u) = x$, proving $\opn{cl}(e_i)\subseteq\tilde{c}_i$.
   
    The closure-finiteness of $\tilde{\varphi}^c_i$, i.e., that $\tilde{\varphi}^c_i(\partial D^{n+1})$ is covered by a finite number of cells of dimension less than $n+1$, follows from the inductive hypothesis.

    The weak topology follows from the inductive hypothesis and the fact that closeness is a local property.
\end{proof}



\begin{lemma}\label{lem:Gcell}
If $B$ is a $CW$-complex,
then the $U$-action on $E$ from 
Lemma~\ref{lem:CellLifting}
is free and cellular,
permuting the cells in $p^{-1}(c)$ for each cell $c \subset B$.
In particular, $E$ is a $U$-$CW$-complex.
\end{lemma}

\begin{proof}
We set $p^{(n)} \coloneqq p|_{E^{(n)}}$, 
consider the covering maps $p^{(n)}\colon E^{(n)} \to B^{(n)}$ and proceed inductively on $n$. 
    
For $n = 0$, this is trivial.
Suppose that $U$ acts on $E^{(n)}$ by permuting the cells $\tilde{c}_e$ for all $e\in F_b$, for any cell $c$ of $B$ of dimension $\leq n$, and $b\in c^{\circ}$. It remains to define the action of $U$ on the set
    \begin{equation*}
        \{\tilde{c}^{\circ}_e\mid e\in F_b\},
    \end{equation*}
     where $c$ is an $(n{+}1)$-cell of $B$ and $b\in c^{\circ}$. This follows immediately from Lemma \ref{lem:CellLifting}, by extending the $U$-action on $F_b$. 
\end{proof}

\section{Milgram's Construction of Classifying spaces}\label{sec:Milgram}
In this appendix we review the Milgram's construction of the classifying space of a topological group $\Gamma$.
Our exposition follows \cite[Chapter II]{adem2013cohomology}.
Let 
\[\sigma^n \coloneqq \{(t_1,\cdots,t_n)\mid 0\leq t_1\leq\cdots\leq t_n\leq 1\}\]
denote the standard topological $n$-simplex.

\begin{definition}\label{def:Milgram}
Define the quotient space 
$\EM \Gamma \coloneqq \left( \coprod_n \Gamma\times \sigma^n\times \Gamma^{\times n}\right)/\!\sim$,
where the equivalence relation $\sim$ is given by 
  \begin{equation*}
    \begin{split}
    &(\gamma, t_1,\cdots,t_n, \gamma_1,\cdots,\gamma_n)\sim \\
      &\begin{cases}
        (\gamma, t_1,\cdots,\hat{t_i}\cdots, t_n, \gamma_1, \cdots,\hat{\gamma_i}, \gamma_i\gamma_{i+1},\cdots,\gamma_n),\ \textrm{if }t_i = t_{i+1}$ or $\gamma_i = 1,\\
        (\gamma \gamma_1, t_2,\cdots,t_n, \gamma_2, \cdots,\gamma_n),\ \textrm{if } t_1 = 0,\\
        (\gamma, t_1,\cdots,t_{n-1}, \gamma_1,\cdots,\gamma_{n-1}),\ \textrm{if } t_n = 1.
      \end{cases}
    \end{split}
  \end{equation*}
  A point in $\EM \Gamma$ is denoted  by $[\gamma, t_1,\cdots,t_n, \gamma_1,\cdots,\gamma_n]$. 
  There is a continuous left $\Gamma$-action on $\EM\Gamma$ defined by
  \begin{equation}\label{eq:Gamma_act_EGamma}
      \gamma'\cdot[\gamma, t_1,\cdots,t_n, \gamma_1,\cdots,\gamma_n]\coloneqq [\gamma'\gamma, t_1,\cdots,t_n, \gamma_1,\cdots,\gamma_n],
  \end{equation}
  and it is easy to check that this action is free. 
  The classifying space $\BM \Gamma$ is defined to be the quotient space $\BM\Gamma \coloneqq \EM\Gamma/\Gamma$. A point of $\BM \Gamma$, the orbit of $[\gamma, t_1,\cdots,t_n, \gamma_1,\cdots,\gamma_n]$, is denoted by 
\[[t_1,\cdots,t_n, \gamma_1,\cdots,\gamma_n].\]
\end{definition}

If $(\Gamma,1)$ is an NDR pair, then the quotient map $\EM \Gamma\to \BM \Gamma$ is a universal principal $\Gamma$-bundle. See \cite[Theorem 8.3]{steenrod1968milgram}.

If $f \colon \Gamma\to \Lambda$ is a continuous homomorphism, then there is an induced map $\EM f \colon  \EM \Gamma\to \EM \Lambda$ defined by
\[\EM f([\gamma, t_1,\cdots,t_n, \gamma_1,\cdots, \gamma_n]) = [f(\gamma),  t_1,\cdots,t_n, f(\gamma_1),\cdots,f(\gamma_n)].\]
Passing to quotients, we have the map $\BM f\colon \BM \Gamma\to \BM \Lambda$,
defined by
\[\BM f([t_1,\cdots,t_n, \gamma_1,\cdots, \gamma_n]) = [t_1,\cdots,t_n, f(\gamma_1), \cdots, f(\gamma_n)].\]
%

\begin{proposition}\label{pro:BMfunctor}
The maps $\EM f$ and $\BM f$ defined above make $\EM$ and $\BM$ functors from the category of topological groups (with continuous homomorphisms) to the category of spaces. 

If, in addition,  $(\Gamma,1)$ and $(\Lambda, 1)$ are NDR pairs, then the pull-back of $\EM \Lambda\to \BM \Lambda$ along $\BM f$ is isomorphic to the principal $\Lambda$-bundle $\EM \Gamma \times_\Gamma \Lambda\to \BM \Gamma$, where $\Gamma$ acts on $\Lambda$ on the left via $f$. We regard the action of $\Gamma$ on $E\Gamma$ as a right action via $x\cdot\gamma:= \gamma^{-1}x$. 
\end{proposition}
\begin{proof}
    The proof of the statement that $\EM f$ and $\BM f$ defined above make $\EM$ and $\BM$ functors from the category of topological groups to space is purely formal and can be found in \cite[II.1]{adem2013cohomology}.

    To show the statement of the second paragraph, consider the map
    \begin{equation*}
        \begin{split}
            E\Gamma\times_{\Gamma}\Lambda &\to B\Gamma\times_{B\Lambda}E\Lambda,\\
            ([\gamma,t_1,\cdots,t_n,\gamma_1,\cdots,\gamma_n], \lambda) &\mapsto\\ ([t_1,\cdots,t_n,\gamma_1,\cdots,&\gamma_n], [\lambda^{-1}f(\gamma), t_1,\cdots,t_n,f(\gamma_1),\cdots,f(\gamma_n))]).
        \end{split}
    \end{equation*}
    It is straightforward to verify that this defines an isomorphism of principal $\Gamma$-bundles over $B_M\Lambda$.
\end{proof}

Consider compactly generated (in the sense of Steenrod \cite[2.1]{steenrod1967convenient}) topological groups $\Gamma$ and $\Lambda$, and their product $\Gamma{\times} \Lambda$ with the compact topology in the sense of \cite[4.1]{steenrod1967convenient}. The projections of
$\Gamma {\times} \Lambda$ onto 
$\Gamma$ and $H$ are continuous homomorphisms and so induce maps
$\BM (\Gamma{\times}\Lambda)\to \BM \Gamma$ and 
$\BM (\Gamma {\times} \Lambda) \to \BM H$, 
which together define the map
$\phi_{\Gamma,\Lambda} \colon \BM(\Gamma {\times} \Lambda) \to \BM \Gamma {\times} \BM H$. 

By \cite[Theorem 1.8, Chapter II]{adem2013cohomology}, the map $\phi_{\Gamma,\Lambda}$ is a homeomorphism, 
to which the inverse is given by 
\begin{equation*}
  ([t_1,\cdots,t_n, \gamma_1,\cdots,\gamma_n], [s_1,\cdots,s_m, \lambda_1,\cdots,\lambda_m])\mapsto [\tau_1\cdots,\tau_{m+n},\theta_1,\cdots,\theta_{m+n}],
\end{equation*}
where $\tau_1\cdots,\tau_{m+n}$ are $t_1,\cdots, t_n, s_1,\cdots,s_m$ in the increasing order, and 
\begin{equation*}
  \theta_k = 
  \begin{cases}
    (\gamma_u,1),\ \tau_k = t_u,\\
    (1,\lambda_v),\ \tau_k = s_v.
  \end{cases}
\end{equation*}
One easily generalizes the above to see that there is a homeomorphism 
\begin{equation}\label{eq:BmG}
  \zeta\colon \prod_i \BM \Gamma_i\to \BM (\prod_i \Gamma_i)
\end{equation}
for compactly generated groups $\Gamma_1,\cdots,\Gamma_n$. In the case $\Gamma_1 = \cdots = \Gamma_n = \Gamma$, the permutation group $S_n$ acts on both sides of \eqref{eq:BmG}, by permuting the components $\BM \Gamma$ on the left, and the components $\Gamma$ on the right. Unfolding the definitions yields the following

\begin{proposition}\label{pro:S_n_equiv}
The homeomorphism $\zeta$ is $S_n$-equivariant. \qed
\end{proposition}





Let $\Gamma = \Lambda \rtimes K$ where $K$ be a discrete group. We define a right $K$-action on $\EM \Lambda$ as follows: For $k\in K$,  and a point $[x, \tau_1,\cdots,\tau_n, x_1\cdots, x_n]$ in $\EM \Lambda$, let
\begin{equation}\label{eq:Gamma_act_BH}
   k\cdot[x, \tau_1,\cdots,\tau_n, x_1\cdots, x_n] \coloneqq [\cc_k( x), \tau_1,\cdots,\tau_n, \cc_k(x)\cdots,  \cc_k(x_n)],
\end{equation}
where $\cc_k(y) = kyk^{-1}$ in $\Lambda \rtimes K$ for $y\in \Lambda$. It is formal to check that this action is well defined. The following lemma is immediate.

\begin{lemma}\label{lem:Gamma_act_BH}
    Let $\Gamma = \Lambda \rtimes K$, where $K$ is discrete and hence $\Lambda \subseteq \Lambda \rtimes K$ is closed.
    \begin{enumerate}
        \item\label{ite:Gamma_act_BH_1} The $\Lambda$-action on $E_M\Lambda$ defined by \eqref{eq:Gamma_act_EGamma} extends uniquely to a $\Gamma$-action, such that
        
        \item when restricted to $K$, is the left $K$-action \eqref{eq:Gamma_act_BH}, and 
     
        \item the $K$-action  \eqref{eq:Gamma_act_BH} on $E_M\Lambda$ induces a $K$-action on $B_M\Lambda$ that coincides with $\BM \cc_k$. \qed
    \end{enumerate}
\end{lemma}


Let $\varphi\colon \Gamma\to K$ be the projection. 

\begin{lemma}\label{lem:htpy_orbit}
  Consider the homotopy orbit space $\BM \Lambda/\!/ K \coloneqq \EM K\times _K\BM \Lambda$. Here we regard the action of $K$ on $\EM K$ as a right action via $x\cdot k \coloneqq k^{-1}\cdot x$.
\begin{enumerate}
        \item\label{ite:htpy_orbit_1} The projection $f \colon \BM \Lambda/\!/ K\to \BM K$ is a fiber bundle with fibers $\BM \Lambda$. 
        \item\label{ite:htpy_orbit_2} Let $\mathscr{H}$ be the local system associated to the fiber bundle of \eqref{ite:htpy_orbit_1} with the stalk over $x\in \BM K$ being $H^*(f^{-1}x;R)\cong H^*(\BM \Lambda;R)$ for a ring $R$. For $k\in K\cong\pi_1(\BM K)$, $\mathscr{H}$ sends $k$ to the automorphism $H^*(\cc_{k^{-1}})$.
        \item\label{ite:htpy_orbit_3} The space $\BM \Lambda/\!/ K$ is the base space of a universal principal $\Gamma$-bundle.
        \item\label{ite:htpy_orbit_4} The projection $f$ is a model for $B\varphi\colon B\Gamma\to BK$.
    \end{enumerate}  
\end{lemma}
\begin{proof}
    The assertion \eqref{ite:htpy_orbit_1} follows from the fact that $\EM K\to \BM K$ is a principal $K$-bundle. 
    
    To establish \eqref{ite:htpy_orbit_2}, observe that $K \cong \pi_1(\BM K)$ and that for $k \in K$, the deck transformation on $\EM K$ coincides with the $K$-action on $\EM K$ defined by \eqref{eq:Gamma_act_EGamma}. For $[a, x] \in \BM \Lambda/\!/ K$, where $a \in \EM K$ and $x \in \EM \Lambda$, we have $[ak, x] = [a, \cc_k(x)]$, and \eqref{ite:htpy_orbit_2} ensues.

    Concerning \eqref{ite:htpy_orbit_3}, consider the space $\EM K \times \EM \Lambda$. Define, for $\gamma \in \Gamma$, $a \in \EM K$, and $b \in \EM \Lambda$, 
    \[
        (a, b) \cdot \gamma \coloneqq (a \cdot \varphi(\gamma), b \cdot \gamma).
    \]
    This action is proven to be free by Lemma \ref{lem:Gamma_act_BH}, \eqref{ite:Gamma_act_BH_1}. The resulting quotient space is homeomorphic to $\BM \Lambda/\!/ K$. Denote the universal principal $K$-bundle $\EM K \to \BM K$ by $p$. Suppose 
    that $\{U_i\}_i$ is trivializing open cover of $p$. Let $p_1\colon \EM K \times_K \BM \Lambda \to \BM K$ be the projection, and define bundle isomorphisms
    \[
        h_i\colon U_i \times \BM \Lambda \cong p_1^{-1}(U_i).
    \]
    Then $p_1^{-1}(U_i)$ possesses an open cover in the form of $\{W_{ij} \coloneqq h_i(U_i \times V_j)\}_j$ where $\{V_j\}_j$ is a trivializing open cover of $E_M \Lambda\to B_M\Lambda$.  Consequently, the quotient map
    \[
        \EM K \times \EM \Lambda \to \BM \Lambda/\!/ K
    \]
    forms a principal $\Gamma$-bundle with a local trivialization $\{W_{ij}\}_{ij}$, 
    which establishes \eqref{ite:htpy_orbit_3}.
    
    Assertion \eqref{ite:htpy_orbit_4} follows directly from \eqref{ite:htpy_orbit_1} and \eqref{ite:htpy_orbit_3}.
    %
    \end{proof}

The following proposition 
simply combines
Parts \ref{ite:htpy_orbit_3} and \ref{ite:htpy_orbit_4} of Lemma~\ref{lem:htpy_orbit}.

\begin{proposition}\label{pro:phi-adapted}
    With the $K$-action above, $\EM \Lambda\to \BM \Lambda$ is a $\varphi$-adapted model for $B\Lambda$. \qed
\end{proposition}

\bibliographystyle{abbrv}
\bibliography{RefWeylGroupInv}

\end{document}